\documentclass[]{interact}

\usepackage{epstopdf}% To incorporate .eps illustrations using PDFLaTeX, etc.
\usepackage[caption=false]{subfig}% Support for small, `sub' figures and tables

\usepackage[numbers,sort&compress]{natbib}% Citation support using natbib.sty
\bibpunct[, ]{[}{]}{,}{n}{,}{,}% Citation support using natbib.sty
\usepackage{enumitem}
\usepackage{hyperref}

\newcounter{dummy}
\theoremstyle{plain}% Theorem-like structures provided by amsthm.sty
\newtheorem{theorem}{Theorem}[section]
\newtheorem{lemma}[theorem]{Lemma}

\newtheorem{proposition}[theorem]{Proposition}

\theoremstyle{definition}
\newtheorem{definition}[theorem]{Definition}
\newtheorem{example}[theorem]{Example}

\theoremstyle{remark}
\newtheorem{remark}{Remark}

\newcommand{\norm}[1]{\ensuremath{\lVert#1\rVert}}
\newcommand{\normone}[1]{\ensuremath{\lVert#1\rVert_1}}
\newcommand{\norminfty}[1]{\ensuremath{\lVert#1\rVert_{\infty}}}

\newcommand{\RR}{\ensuremath{\mathbb{R}}}
\newcommand{\NN}{\ensuremath{\mathbb{N}}}

\newcommand{\cU}{\mathcal{U}}

\newcommand{\mM}{Z}

\newcommand{\Id}{\ensuremath{\mathrm{Id}}}

\newcommand{\taumin}{\ensuremath{\tau_{\mathrm{min}}}}
\newcommand{\taumax}{\ensuremath{\tau_{\mathrm{max}}}}

\newcommand{\dd}{\ensuremath{\mathrm{d}}}
\newcommand{\DD}{\ensuremath{\mathrm{D}}}

\DeclareMathOperator{\inte}{Int}
\DeclareMathOperator{\dist}{dist}
\DeclareMathOperator{\cl}{cl}
\DeclareMathOperator{\diag}{diag}

\DeclareMathOperator{\card}{card}

\DeclareMathOperator{\bd}{bd}

\newcommand{\ev}{\ensuremath{\mathbf{e}}}

\DeclareFontFamily{U}{mathb}{\hyphenchar\font45}
\DeclareFontShape{U}{mathb}{m}{n}{
<-6> mathb5 <6-7> mathb6 <7-8> mathb7
<8-9> mathb8 <9-10> mathb9
<10-12> mathb10 <12-> mathb12
}{}
\DeclareSymbolFont{mathb}{U}{mathb}{m}{n}
\DeclareMathSymbol{\llcurly}{\mathrel}{mathb}{"CE}
\DeclareMathSymbol{\ggcurly}{\mathrel}{mathb}{"CF}

\makeatletter
\DeclareFontFamily{OMX}{MnSymbolE}{}
\DeclareSymbolFont{MnLargeSymbols}{OMX}{MnSymbolE}{m}{n}
\SetSymbolFont{MnLargeSymbols}{bold}{OMX}{MnSymbolE}{b}{n}
\DeclareFontShape{OMX}{MnSymbolE}{m}{n}{
    <-6>  MnSymbolE5
   <6-7>  MnSymbolE6
   <7-8>  MnSymbolE7
   <8-9>  MnSymbolE8
   <9-10> MnSymbolE9
  <10-12> MnSymbolE10
  <12->   MnSymbolE12
}{}
\DeclareFontShape{OMX}{MnSymbolE}{b}{n}{
    <-6>  MnSymbolE-Bold5
   <6-7>  MnSymbolE-Bold6
   <7-8>  MnSymbolE-Bold7
   <8-9>  MnSymbolE-Bold8
   <9-10> MnSymbolE-Bold9
  <10-12> MnSymbolE-Bold10
  <12->   MnSymbolE-Bold12
}{}

\let\llangle\@undefined
\let\rrangle\@undefined
\DeclareMathDelimiter{\llangle}{\mathopen}%
                     {MnLargeSymbols}{'164}{MnLargeSymbols}{'164}
\DeclareMathDelimiter{\rrangle}{\mathclose}%
                     {MnLargeSymbols}{'171}{MnLargeSymbols}{'171}
\makeatother

\DeclareFontFamily{U}{mathx}{\hyphenchar\font45}
\DeclareFontShape{U}{mathx}{m}{n}{
      <5> <6> <7> <8> <9> <10>
      <10.95> <12> <14.4> <17.28> <20.74> <24.88>
      mathx10
      }{}
\DeclareSymbolFont{mathx}{U}{mathx}{m}{n}
\DeclareFontSubstitution{U}{mathx}{m}{n}
\DeclareMathAccent{\widecheck}{0}{mathx}{"71}
\DeclareMathAccent{\wideparen}{0}{mathx}{"75}

\makeatletter
\newcommand*{\textoverline}[1]{$\overline{\hbox{#1}}\m@th$}
\makeatother

\makeatletter
\newcommand\myitem[1][]{\item[#1]\refstepcounter{dummy}\def\@currentlabel{#1}}
\makeatother

\begin{document}

\articletype{}

\title{Existence of the carrying simplex for a retrotone map}

\author{
\name{Janusz Mierczyński\textsuperscript{a}\thanks{CONTACT J. Mierczyński. Email: mierczyn@pwr.edu.pl} and Stephen Baigent\textsuperscript{b}}
\affil{\textsuperscript{a}Faculty of Pure and Applied Mathematics, Wrocław University of Science and Technology, Wybrzeże Wyspiańskiego 27, PL 50-370 Wrocław, Poland; \textsuperscript{b}Department of Mathematics, University College London, Gower Street, London WC1E 6BT, United Kingdom of Great Britain and Northern Ireland}
}

\maketitle

\begin{abstract}
We present a dynamical approach to the study of unordered, attracting manifolds of retrotone maps commonly known as carrying simplices. Our approach is novel in that it uses  the radial representation of unordered manifolds over the probability simplex coupled with distances between these manifolds measured by way of the Harnack and Hausdorff metrics. We establish Kuratowski convergence of radial representations of unordered manifolds to a unique function which then provides the locally Lipschitz radial representation of the carrying simplex.
\end{abstract}


This is an Accepted Manuscript of an article published by Taylor \& Francis in JOURNAL OF DIFFERENCE EQUATIONS AND APPLICATIONS on 23 November 2023, available at: \href{https://doi.org/10.1080/10236198.2023.2285394}{https://doi.org/10.1080/10236198.2023.2285394}.

\section*{Introduction}
There is a large class of maps widely used  to study discrete-time population dynamics that preserve all faces of the first orthant of Euclidean space that are known as {\em Kolmogorov maps}. By varying the functional forms of these maps, all the possible types of species interactions for populations where there is no overlap of generations can be modelled, including predator-prey, mutualism and competition to name a few. Here we are interested in a special {\em subclass} of those Kolmogorov maps that feature only competition which, following  recent research trends, we call {\em retrotone}. While much of the existing research uses  {\em competitive map} where we use {\em retrotone map} - indeed the distinction is not crucial (necessary) for continuous time analogues - we prefer to use `retrotone' to describe maps that represent competitive interactions that {\em also} have the specific feature that they preserve a convex cone backwards in time. The Leslie--Gower map is an example of a competition model that is globally a retrotone map \cite{JN-JMB2017}, but not all competitive maps are retrotone, even on their global attractor; the planar Ricker map discussed later in Subsection~\ref{subsect:2DRicker} is a classic example of a map that represents competitive interactions, but is only retrotone for a limited set of parameter values.

The focus of the present work is the {\em carrying simplex} which is a well-studied feature of population models with competitive interactions  \cite{dMS, SmithJDE, HirschCS, JandW, D-W-Y, HirschJBD, RH}. The carrying simplex is a Lipschitz, codimension-one and compact invariant manifold that attracts all nonzero points, and that is unordered, which means that the carrying simplex is non-increasing in each coordinate direction. As we discuss later there are several ways of characterising the carrying simplex, including as the set of nonzero points with globally defined and bounded forward and backward orbits, or as the relative boundary of the global attractor of bounded sets.  The presence of the carrying simplex (which is asymptotically complete) means that the limiting dynamics can be studied on the carrying simplex which gives rise to a system of one fewer degree of freedom. This has been exploited by a number of authors to study the global dynamics of both continuous and discrete-time population models. For example, \cite{JN-JMB2017, Niu2016, Gyll} study global dynamics of maps with a carrying simplex by way of an index theorem. Baigent and Hou \cite{BH2012,BH2017} utilise the carrying simplex in both continuous and discrete time models to construct Lyapunov functions on forward invariant sets, and Montes de Oca and Zeeman \cite{D_Oca} use the carrying simplex concept iteratively to reduce a continuous-time competition model to an easily solved one-dimensional model.

That there exists such an invariant manifold under ecologically reasonable assumptions is intuitive. We impose that the origin is unstable, which means that small population densities grow; all the species cannot simultaneously go extinct. We also require that the per-capita growth rate of a given species decreases with increasing densities of all species (although we relax this to non-increasing for some densities later). This can be thought of as modelling competition for resources. In differential equation models for competition, these assumptions are typically sufficient for a carrying simplex to exist \cite{HirschCS,Hou2020}, but not in the discrete time case. We follow other authors  \cite{JandW,HirschJBD,RH,Hou21} and add further conditions that render the map retrotone. Roughly speaking the retrotone property, which can be imposed through spectral properties of the derivative of the map  \cite{HirschJBD,RH,Hou21}, aside from modelling competitive interactions, puts restrictions on the maximal change in total population density over one generation, i.e. large radial changes under one application of the map are not permitted. The retrotone properties of the map render the carrying simplex unordered, so that it projects radially onto the probability simplex. Hence on the carrying simplex, given the frequency of the species, i.e. a point in the probability simplex, the radial coordinate, which is the total population density, is determined.

We argue that the radial representation, in which phase space is the cartesian product of the probability simplex and the positive real line, is a natural coordinate choice, and it is the main set of coordinates that we use here. In this description the carrying simplex is just the graph of a continuous functions, locally Lipschitz on the interior of the phase space, over the probability simplex.  However, working with the radial representation presents technical difficulties at the boundary of phase space where derivatives can become unbounded. To resolve this issue we use the Kuratowski metric to establish Hausdorff convergence of the unordered manifolds generated by the graph transform approach. We generate one increasing and one decreasing sequence of Kuratowski convergent sequences of unordered manifolds and then we utilise the Harnack metric to show that the two limits are identical and identified as the carrying simplex.

\bigskip

The paper is organised as follows.  In Section \ref{sec:sec-notation} we introduce our notation and give important definitions, such as for cone-orderings, unordered and weakly unordered sets, retrotone maps and weakly retrotone maps, and attractors of various classes of sets.   In~particular, a definition of the carrying simplex is proposed (Definition~\ref{cUintocU}).  Since no standard definition has been so far settled on, we choose to base our proposal on definitions in~\cite{HirschJBD} and~\cite{RH}, perhaps with more dynamical flavour added (property~\ref{P-attraction}).  We mention also some additional properties, \ref{P-order-convex} up to \ref{P-Lipschitz}, that have appeared in some earlier papers.  We will see later that all those properties are satisfied under our assumptions.  Section \ref{sec_existence} is concerned with proof of the existence of the carrying simplex. Subsection \ref{subsect:axes} deals with the dynamics of the map restricted to the boundary, and  is crucial for establishing the existence of a bounded rectangle  on which the map is retrotone/weakly retrotone and which contains the unique compact attractor of bounded sets. Conditions under which the map is retrotone/weakly retrotone  on the bounded rectangle are established in Subsection \ref{subsect:retrotone} and Subsection \ref{subsect:Lambda-attracts} characterises the compact attractor of bounded sets of the map.  In Subsection~\ref{subsect:construction-of-CS} we construct the carrying simplex/weak carrying simplex from sequences of unordered/weakly unordered manifolds using the graph transform and Kuratowski convergence radial coordinates, as well as give some of its properties.  Attractor--repeller pairings are used to establish further properties of the carrying simplex/weak carrying simplex in Subsections~\ref{subsect:back-to-C_+}, and Subsections \ref{subsect:asymptotic} and~\ref{subsect:Lipschitz} deal with asymptotic completeness and a Lipschitz representation of the carrying simplex. In Section \ref{sect:examples} we discuss some examples that illustrate the main ideas of the paper. Finally in Section \ref{subsect:ODEs} we show how to use our results for retrotone [weakly retrotone] maps to derive conditions for the existence of a carrying simplex [weak carrying simplex] for competitive systems of ordinary differential equations.

\section{Notation and definitions}\label{sec:sec-notation}
\noindent We will need the following notation and definitions.

$\lVert \cdot \rVert$ stands for the Euclidean norm in $\mathbb{R}^d$, $\normone{\cdot}$ denotes the corresponding $\ell_1$\nobreakdash-\hspace{0pt}norm in $\RR^d$, and $\norminfty{\cdot}$ denotes the corresponding $\ell_{\infty}$\nobreakdash-\hspace{0pt}norm in $\RR^d$, $d \ge 1$.  For $x,y\in \RR^d$, $x\cdot y$ denotes the standard inner product $\sum\limits_{i=1}^d x_i y_i$ in $\RR^d$.

For $x \in \RR^d$ and $A \subset \RR^d$ we write
\begin{equation*}
    \dist(x, A) := \inf\,\{\, \norm{x - a} : a \in A \,\},
\end{equation*}
and for nonempty compact $A, B \subset \RR^d$ we denote by $d_H(A, B)$ their \emph{Hausdorff distance},
\begin{equation*}
    d_H(A, B) := \max\{ \sup\,\{\, \dist(a, B) : a \in A \,\}, \sup\,\{\, \dist(b, A) : b \in B \,\}\}.
\end{equation*}

$C_{+} := \{x\in \RR^d: x_i\geq 0, \ i=1,\ldots,d\}$, and $C_{++} := \{x\in \RR^d: x_i > 0, \ i=1,\ldots,d\}$ will denote convex cones. $C_+$ is often referred to as the \emph{first orthant} and $C_{++}$ is its interior.  $\partial C_{+} = C_{+} \setminus C_{++}$ is called the \emph{boundary} of $C_{+}$ (indeed, it is the boundary of $C_{+}$ in $\RR^d$).

We denote $\NN =  \{0,1, 2, 3, \dots\}$. For a subset $I \subset \{1, \ldots,d\}$, let $\RR^d_{I} := \{\, x \in \RR^d: x_i = 0$~for~all~$i \in \{1, \ldots,d\} \setminus I\,\}$, let
$(C_I)_{+} := C_{+} \cap \RR^d_{I}$ denote a \emph{$k$\nobreakdash-\hspace{0pt}dimensional face} of $C_{I}$, where $k = \card{I}$, and let $(C_I)_{++} := \{\, x \in (C_I)_{+}: x_i > 0$~for~all~$i \in I\,\}$ denote the \textit{relative interior} of $(C_I)_{+}$.  $\partial (C_I)_{+} := (C_I)_{+} \setminus (C_I)_{++}$ is the \textit{relative boundary} of $(C_I)_{+}$.  A $1$\nobreakdash-\hspace{0pt}dimensional face is referred to as an \textit{$i$\nobreakdash-\hspace{0pt}th axis}, where $I = \{i\}$. Instead of~$(C_{\{i\}})_+$, etc., we write $(C_{i})_+$, etc.  For $x \in C_{+}$ let $I(x)$ stand for the (unique) subset of $\{1, \ldots,d\}$ such that $x \in (C_{I(x)})_{++}$ (often $I(x)$ is called the \emph{support} of $x$).

$\bar{A}$ denotes the closure of $A$.  Let $D$ be a closed subset of $C_+$.  $A \subset D$ is said to be \emph{relatively open} in $D$ if there is an open subset $U$ of $\RR^d$ such that $A = U \cap D$.  $\inte_{D}{A}$ (called the \emph{relative interior} of $A$ in $D$) stands for the largest subset of $A$ that is relatively open in $D$, and $\bd_{D}{A} = \bar{A} \setminus \inte_{D}{A}$.  For $x \in D$, we say $U \subset D$ is \emph{relative neighbourhood} of $x$ in $D$ if there exists a neighbourhood $V$ of $x$ in $\RR^d$ such that $U = V \cap D$.  Neighbourhoods/relative neighbourhoods are tacitly assumed to be open/relatively open.

For $x, y \in C_{+}$, we write $x \le  y$ if $x_i \le  y_i$ for all $i=1,\ldots,d$, and $x \ll  y$ if $x_i < y_i$ for all $i=1,\ldots,d$. If $x \le  y$ but $x \ne  y$ we write $x <  y$. The reverse relations are denoted by $\geq, >,\gg$. Two points $x,y\in C_+$ are said to be \emph{order-related} if either $x\leq y$ or $y\leq x$.

For $x, y \in C_{+}$ such that $x \le y$ we define the \textit{order interval} as
\begin{equation*}
   [x, y] := \{\, z \in C_{+} : x \le z \le y \,\}.
\end{equation*}
We say $B \subset C_{+}$ is \textit{order~convex} if for any $x, y \in B$ with $x \le y$ one has $[x, y] \subset B$.

Let $\emptyset \ne I \subset \{1, \dots, d\}$.  For $x, y \in (C_{I})_+$, we write $x \le_I  y$ if $x_i \le  y_i$ for all $i \in I$, and $x \ll_I  y$ if $x_i < y_i$ for all $i \in I$.  If $x \le_I  y$ but $x \ne  y$ we write $x <_I  y$. The reverse relations are denoted by $\geq_I, >_I,\gg_I$.

$\Delta := \{x \in C_{+}:\sum\limits_{i=1}^d x_i=1\}$ denotes the standard probability $(d - 1)$\nobreakdash-\hspace{0pt}simplex.  For $I \subset \{1, \ldots,d\}$, $\Delta_{I} := \Delta \cap (C_I)_{+}$.

For $I \subset \{1, \dots, d\}$, by $\pi_I$ we understand the orthogonal projection of $C_{+}$ onto $(C_I)_{+}$.

$\Pi$ denotes the orthogonal projection along $\ev$ onto $V := \ev^{\perp}$,  $\Pi u = u - (u \cdot \hat{\ev}) \hat{\ev} = u - (u \cdot \ev) \ev/d$. Here $\ev_i$ is the vector with a one at the $i$\nobreakdash-\hspace{0pt}th position and zeros elsewhere and $\ev =\sum\limits_{i=1}^d\ev_i$, $\hat{\ev}=\ev/\sqrt{d}$.

An important concept needed to describe the carrying simplex is that of unordered sets. We also use a weaker notion of weakly unordered sets (introduced (but not in name) in \cite[Rem.~2.1(f)]{Hou21}):
\begin{definition}
\label{def:unordered}
A set $B \subset C_+$ is said to be \emph{unordered} if no two distinct points of $B$ are ordered by the $<$  relation.

A set $B \subset C_+$ is said to be \emph{weakly unordered} if for any $\emptyset \ne I \subset \{1, \ldots, d\}$ no two distinct points of $B \cap (C_I)_{+}$ are ordered by the $\ll_I$  relation.
\end{definition}

\begin{example}
\label{ex:simplex}
    \normalfont The standard probability simplex $\Delta \subset C_{+}$ is unordered.
\end{example}

\begin{example}
\label{ex:weakly-unordered-not-ordered}
    \normalfont Let $d > 1$.  For $a > 0$ consider the set
    \begin{equation*}
    \begin{aligned}
        H(a) & := \bd_{C_{+}}[0, a \ev]
        \\
        & = \{\, x \in C_{+} : \forall i \in \{1, \ldots, d\}\, 0 \le x_i \le a, \ \exists k \in \{1, \ldots, d\} \text{ such that } x_k = a \,\}.
    \end{aligned}
    \end{equation*}
    The set $H(a)$ is weakly unordered.  To show this, suppose to the contrary that there are a nonempty $I \subset \{1, \dots, d\}$ and $x, y \in H(a) \cap (C_I)_{+}$ with $x \ll_I y$, which means that $x_i < y_i$ for all $i \in I$ and $x_j = y_j = 0$ for all $j \in \{1, \dots, d\} \setminus I$.  But then $x_i < a$ for all $1 \le i \le d$, so $x$ cannot belong to $H(a)$. On the other hand, $H(a)$ is not unordered:  for example, $a \ev_1 < a \ev$ and both belong to $H(a)$.
\end{example}

Another important concept in the theory of carrying simplices is that of a retrotone map. Retrotonicity is the property that ensures that ordered points are ordered along backward orbits.
\begin{definition}
\label{def:competitive}
A map $F \colon C_{+} \to C_{+}$ is {\em retrotone} in a subset $B \subset C_{+}$, if, for all $x, y \in B$ with $ F(x) < F(y)$, one has that $x_i < y_i$ provided $y_i > 0$.

A map $F \colon C_{+} \to C_{+}$ is \emph{weakly retrotone} in $B \subset C_{+}$ provided for all $I \subset \{1, \ldots, d\}$ and any $x, y \in B$, if
  \begin{equation*}
    F(x) < F(y) \quad \text{and} \quad F_i(x) < F_i(y) \text{ for all } i \in I
  \end{equation*}
  then
  \begin{equation*}
    x < y \quad \text{and} \quad x_i < y_i \text{ for all } i \in I.
  \end{equation*}
\end{definition}
The name `retrotone' in the context of competitive maps appeared first in~\cite{HirschJBD} and later in~\cite{RH}, \cite{Hou21} (however, notice that in~\cite{HirschJBD} our `retrotone' is called `strictly retrotone').   The term `weakly retrotone' was introduced in~\cite{Hou21}.

\bigskip
From now on, we assume that $F =(F_1, \ldots, F_d) \colon C_{+} \to C_{+}$ is a continuous map.  We will consider the dynamical system $(F^n)_{n=0}^{\infty}$ on $C_{+}$, where $F^0 = \Id_{C_{+}}$ and for convenience we will write $F^n = (F_1^n, \ldots, F_d^n)$.

For $x \in C_{+}$ we denote its \textit{forward orbit}, $O^{+}(x)$, as
\begin{equation*}
    O^{+}(x) := \{\, F^n(x) : n \in \NN \, \}.
\end{equation*}

$A \subset C_{+}$ is \textit{forward invariant} if $F(A) \subset A$, and \textit{invariant} if $F(A) = A$.

By a \textit{backward orbit} of $x \in C_{+}$ we understand a set $\{\dots, x_{-n-1}, x_{-n}, \dots, \break x_{-2}, x_{-1}, x_{0}\}$ such that $x_0 = x$ and $x_{-n} = F(x_{-n-1})$ for all $n \in \NN$ (so as to allow for noninvertible maps). A \textit{total orbit} of $x \in C_{+}$ is the union of a backward orbit of $x$ and the forward orbit, $O^{+}(x)$.

\medskip

Following the terminology as in~\cite{S-T} we say that $B \subset C_{+}$ \textit{attracts} the set $A \subset C_{+}$ if for each $\epsilon > 0$ there is $n_0 \in \NN$ such that $\dist(F^n(x), B) < \epsilon$ for all $n \ge n_0$ and all $x \in A$.

For a set $A \subset C_{+}$ define its \emph{$\omega$\nobreakdash-\hspace{0pt}limit set} as
\begin{equation*}
  \omega(A) := \bigcap\limits_{k=0}^{\infty} \overline{\left(  \bigcup\limits_{l=k}^{\infty} F^l(A) \right)}.
\end{equation*}

The dynamical system $(F^n)_{n=0}^{\infty}$ is said to be \textit{asymptotically compact on $A \subset C_{+}$} if for any sequence $(n_k)_{k = 0}^{\infty}$, $n_k \to \infty$, and any sequence $(x_k)_{k=0}^{\infty} \subset A$, the sequence $(F^{n_k}(x_k))_{k = 0}^{\infty}$ has a convergent subsequence.
\begin{lemma}{\normalfont{\cite[Prop.~2.10]{S-T}}}
\label{lm:attracting-equivalent}
    Assume that $(F^n)_{n=0}^{\infty}$ is asymptotically compact on a nonempty $A$.  Then a compact $B$ attracts $A$ if and only if $\omega(A) \subset B$.
\end{lemma}

Let $B \subset C_{+}$ be forward invariant.  By the \textit{compact attractor of bounded sets in $B$} we mean a nonempty compact invariant set $\Gamma \subset B$ that attracts any bounded $A \subset B$.  Such a set is unique (see~\cite[Thm.~2.19 on p.~37]{S-T}).  By the \textit{compact attractor of neighbourhoods of compact sets in $B$} we mean a nonempty compact invariant set $\Gamma \subset B$ such that for any compact $A \subset B$ there is a relative neighbourhood $U$ of $A$ in $B$ such that $\Gamma$ attracts $U$.  Such a set $\Gamma$ is unique (see~\cite[Thm.~2.19 on p.~37]{S-T}).

For $B = C_{+}$ we say simply \textit{compact attractor of bounded sets}, which is the same as the compact attractor of neighbourhoods of compact sets.

\begin{proposition}{\normalfont{\cite[Thm.~2.20 on p.~37]{S-T}}}
\label{prop:attractor-characterization}
    Let $B \subset C_{+}$ be forward invariant.  The compact attractor of bounded sets is characterised as the set of all $x \in B$ having bounded total orbits.
\end{proposition}

A compact set $B \subset C_{+}$ is called a \emph{uniform repeller} in $C_{+}$ if there is $\epsilon > 0$ such that $\liminf\limits_{n \to \infty} \dist(F^n(x), B) \ge \epsilon$ for any $x \in C_{+} \setminus B$.  It is equivalent to the existence of a neighbourhood $U$ of $B$ in $C_{+}$ such that for each $x \in U \setminus B$ there is $n_0 \in \NN$ with the property that $F^n(x) \notin U$ for all $n \ge n_0$ (see~\cite[Rem.~5.15 on p.~136]{S-T}).

\medskip

We say that $F \colon C_{+} \to C_{+}$ is a \emph{Kolmogorov map} if $F = \diag[\mbox{id}]f$ where $f \colon C_+ \to C_+$.

\begin{definition}
\label{def_CS}
The \emph{carrying simplex} \textup{[}resp.\ \emph{weak carrying simplex}\textup{]} for a Kolmogorov map $F \colon C_{+} \to C_{+}$ is a subset $\Sigma \subset C_{+}$ with the following properties:
\begin{enumerate}[label=\textup{(\roman*)},ref=\textup{(\roman*)}]
\item
\label{P-unordered}
    $\Sigma$ is an unordered \textup{[}resp.\ weakly unordered\textup{]} subset of the \textup{(}unique\textup{)} compact attractor of bounded sets $\Gamma$.
\item
\label{P-radial-proj}
    $\Sigma$ is homeomorphic via radial projection to the $(d-1)$\nobreakdash-\hspace{0pt}dimensional standard probability simplex $\Delta$.
\item
\label{P-invariant}
    $F(\Sigma) = \Sigma$ and $F{\restriction}_{\Sigma} \colon \Sigma \to \Sigma$ is a homeomorphism.
\item
\label{P-attraction}
    $\Sigma$ attracts any bounded $A \subset C_{+}$ with $0 \notin \bar{A}$.
\item
\label{P-asymptotic}
    For any $x \in C_{+} \setminus \{0\}$ there is $y \in \Sigma$ such that $\displaystyle \lim\limits_{n\to +\infty} \norm{F^n(x) - F^n(y)} = 0$ \textup{(}this property is called \emph{asymptotic completeness} or \emph{asymptotic phase}\textup{)} \textup{[}resp.\ for any $x \in \Gamma \setminus \{0\}$ there is $y \in \Sigma$ such that $\displaystyle \lim\limits_{n\to +\infty} \norm{F^n(x) - F^n(y)} = 0$\textup{]}.
\end{enumerate}
\end{definition}
For the carrying simplex, the unorderedness in \ref{P-unordered} appears in \cite{HirschCS} (but not explicitly in \cite{Zee02} or \cite{Zee-Zee02}) in the case of competitive systems of ODEs, and in \cite{SmithJDE}, \cite{JandW}, \cite{RH} in the discrete time case.  For the weak carrying simplex, the weak unorderedness in \ref{P-unordered} appears in \cite{Hou2020} in the case of competitive systems of ODEs, and in \cite{Hou21} in the discrete time case.  The fact that $\Sigma$ is contained in the compact attractor of bounded sets $\Gamma$ is seldom explicitly mentioned (as in \cite{SmithJDE}), but it follows from dissipativity assumed in other papers.

Property \ref{P-radial-proj} is usually mentioned explicitly (but in \cite{RH} the homeomorphism is defined in another way).

Invariance in \ref{P-invariant} is always mentioned.

To our knowledge, the only place where property \ref{P-attraction} has been \emph{explicitly} stated is \cite[p.~291]{Hou21}.  Indeed, in many papers it can be inferred from the property that the carrying simplex is obtained therein as the upper boundary of the repulsion basin of $\{0\}$, see, e.g., \cite{SmithJDE, JandW}.

Property \ref{P-asymptotic} is present everywhere, starting from~\cite[Lem.~4.4]{HirschCS}.

\smallskip
Below we mention some additional properties.
\begin{enumerate}[label=\textup{(\roman*)},ref=\textup{(\roman*)}]
\setcounter{enumi}{5}
    \item
    \label{P-order-convex}
    {\itshape $\Sigma$ is the boundary \textup{(}relative to $C_{+}$\textup{)} of $\Gamma$ and $\Gamma$ is order convex.  In~particular, $\Gamma = \{\alpha x :  \alpha \in [0, 1], \ x \in \Sigma\}$.}
    \item
    \label{P-attractor}
    {\itshape $\Gamma \setminus \Sigma = \{\alpha x :  \alpha\in [0, 1), x \in \Sigma\}$ is characterised as the set of all those $x \in C_{+}$ that have a backward orbit $\{\ldots, x_{-2}, x_{-1}, x\}$ with $\lim\limits_{n \to \infty} x_{-n} = 0$.}
    \item
    \label{P-characterization}
    {\itshape $\Sigma$ is characterised as the set of all $x \in C_{+}$ having total orbits that are bounded and bounded away from $0$.}
    \item
    \label{P-Lipschitz}
    {\itshape The inverse $(\Pi{\restriction}_{\Sigma})^{-1}$ of the orthogonal projection of $\Sigma$ along $\ev$ is Lipschitz continuous.}
\end{enumerate}
As stated earlier, in the existing papers \ref{P-order-convex} is one of the main ingredients in the proof of the existence of the carrying simplex (cf., for example, \cite[Thm.~6.1]{RH}).

The characterisations given in \ref{P-attractor} and \ref{P-characterization} appear in~\cite{RH}.  In the present paper they follow from abstract dynamical systems theory.

\ref{P-Lipschitz} occurred first in \cite{HirschCS}.  Since then it has seldom appeared.

It has been frequently mentioned that the carrying simplex is unique.  It follows from the conjunction of \ref{P-unordered} and \ref{P-radial-proj}, or from \ref{P-attraction} (using forward invariance of all faces), however in our approach it is simpler to use the additional property \ref{P-characterization}.

\section{Existence of a carrying simplex} \label{sec_existence}

We give two sets of assumptions which guarantee the existence of the carrying simplex (Theorem~\ref{prop1_steve}).

In the case of the first set, \ref{AS-C^0} up to \ref{AS-3}, we start by assuming the existence of some bounded rectangle $\Lambda$ such that $F$ restricted to $\Lambda$, $F{\restriction}_\Lambda$, is [weakly] retrotone (assumption \ref{AS4}).  We work this way because not all maps with a carrying simplex are retrotone on all of $C_+$ (an example is given in Subsection~\ref{subsect:2DRicker}).  We prove then that $F{\restriction}_\Lambda$, satisfies Definition \ref{def_CS} with instances of $C_+$ replaced by $\Lambda$.  As [weak] retrotonicity is often difficult to prove directly, in Subsection~\ref{subsect:retrotone} we give sufficient conditions, formulated in terms of the spectral radius of some matrix, for weak retrotonicity or retrotonicity to be satisfied.  Those conditions are fulfilled for many discrete time competition models, as explained in Section~\ref{sect:examples}.

The second set of assumptions, with \ref{AS-3} replaced by \ref{AS3-mod}, covers the case when $F$ is the time-one map of a competitive system of ODEs, and we utilise it in Subsection \ref{subsect:ODEs} to recover well-known conditions for the existence of a carrying simplex in a system of competitive ODEs.  Then retrotonicity on the whole of $C_{+}$ is a consequence of the Müller--Kamke theory \cite{muller,kamke}.  On the other hand, \ref{AS-3} may be  difficult to check, so it is replaced by \ref{AS3-mod}.  Now, the role of $\Lambda$ can be played by any sufficiently large rectangle.

The main part of the present section, Subsection~\ref{subsect:construction-of-CS}, contains a proof of the existence of a set $\Sigma$ satisfying \ref{P-unordered}, \ref{P-radial-proj} and \ref{P-invariant} in Definition~\ref{def_CS}.   Also, the additional property~\ref{P-order-convex} is proved there.

In the second step (subsection~\ref{subsect:back-to-C_+}) we show that all points in $C_+$ eventually enter and stay in $\Lambda$, so that, with the help of the dynamical systems theory, $\Sigma$ actually attracts any bounded $A \subset C_{+}$ with $0 \notin \bar{A}$ (property \ref{P-attraction}).  The map is not required to be retrotone outside $\Lambda$.  As a by-product we obtain the satisfaction of the additional properties \ref{P-attractor} and \ref{P-characterization}.

Subsection~\ref{subsect:asymptotic} contains a proof of property~\ref{P-asymptotic} (so, only at that point can $\Sigma$ be legitimately called the carrying simplex).  A proof of the additional property~\ref{P-Lipschitz} is given in~Subsection~\ref{subsect:Lipschitz}.

 \medskip
 We make the following assumptions:

 Let $F \colon C_{+} \to C_{+}$ be a Kolmogorov map $F := \diag [\mbox{id}] f$ where $f \colon C_+ \to C_+$ satisfies
\begin{enumerate}[label=\textup{A\arabic*},ref=\textup{A\arabic*},align=left]
\item \label{AS-C^0}
\quad $f$ is continuous, with $f(x) \gg 0$ for all $x \in C_{+}$;
\item \label{AS-e}
\quad  $f_i(\ev_i)= 1$,  $i=1,\ldots,d$;
\item \label{AS4}
\quad
there exists $\varkappa > 0$ such that, putting  $\Lambda := [0, (1+\varkappa)\ev]$,
\begin{enumerate}[label=\textup{A\arabic{enumi}-\alph{enumii}},ref=\textup{A\arabic{enumi}-\alph{enumii}}]
    \item \label{AS-homeo}
    \quad $F{\restriction}_{\Lambda} \colon \Lambda \to F(\Lambda)$ is a local homeomorphism,
    \item \label{AS-weak-retro}
    \quad $F$ is weakly retrotone in $\Lambda$,
\end{enumerate}
\item \label{AS-3}
\quad  for any $x, y \in C_{+}$, if $x < y$ then
\begin{enumerate}[label=\textup{A\arabic{enumi}-\alph{enumii}},ref=\textup{A\arabic{enumi}-\alph{enumii}}]
    \item \label{AS-3-1}
    \quad $f_i(x) \ge f_i(y)$ for all  $i$, and
    \item \label{AS-3-competitive}
    \quad $f_i(x) > f_i(y)$ for those $i$ for which $x_i < y_i$;
\end{enumerate}

\end{enumerate}
In \ref{AS-homeo} by a \textit{local homeomorphism} we mean that for each $x \in \Lambda$ there exist a relative, in $\Lambda$,  neighbourhood $U$ of $x$ and a relative in $F(\Lambda)$ neighbourhood $V$ of $F(x)$ such that $F{\restriction}_U \colon U \to V$ is a homeomorphism.

\begin{remark}
\label{rm:invariance}
    From \ref{AS-C^0} it follows that for any $\emptyset \ne I \subset \{1, \ldots, d\}$ there holds $F((C_I)_{++}) \subset (C_I)_{++}$ and $F^{-1}((C_I)_{++}) \subset (C_I)_{++}$.
\end{remark}
\begin{remark}
In \ref{AS-e}, we assume that each axis has a fixed point of $F$ at $\ev_i$, but by rescaling we may deal with fixed point at $q_i \ev_i$ for any set of $q_i>0$, $i=1,\ldots,d$.
\end{remark}
Put $\Lambda' := \Lambda \setminus \{0\}$.

\begin{lemma}
\label{lm:AS-homeo}
    Assume~\ref{AS-C^0} and~\ref{AS-homeo}.  Then $F{\restriction}_{\Lambda}$ is a homeomorphism onto its image.
\end{lemma}
\begin{proof}
  By \ref{AS-homeo}, the map $F{\restriction}_{\Lambda}$ is a local homeomorphism.  Since $\Lambda$ is compact, $F{\restriction}_{\Lambda}$ is a proper map.  Further, $\Lambda$ being connected, its image $F(\Lambda)$ is connected, and, as it does not contain critical points (i.e., points where $F$ is not locally invertible), we can apply \cite[Lem.~2.3.4]{ChHa} to conclude that the cardinality of $(F{\restriction}_{\Lambda})^{-1}(y)$ is constant for all $y \in F(\Lambda)$.  As it follows from the Kolmogorov property and~\ref{AS-C^0} that $\card{F^{-1}(0)} = 1$, $F{\restriction}_{\Lambda}$ is injective, so, being continuous from a compact space, is a homeomorphism onto its image.
\end{proof}
\noindent (For similar reasoning see~\cite[Lem.~4.1]{RH}).

Sometimes instead of~\ref{AS4}--\ref{AS-3} we make the following stronger assumptions:
\begin{enumerate}[label=\textup{A\arabic*$'$},ref=\textup{A\arabic*$'$}]
\setcounter{enumi}{2}
    \item \label{AS4-strong}
    \quad
    there exists $\varkappa > 0$ such that, putting $\Lambda := [0,(1+\varkappa)\ev]$,
        \begin{enumerate}[label=\textup{A\arabic{enumi}$'$-\alph{enumii}},ref=\textup{A\arabic{enumi}$'$-\alph{enumii}}]
            \item \label{AS-homeo-strong}
            \quad $F{\restriction}_{\Lambda} \colon \Lambda \to F(\Lambda)$ is a local homeomorphism, and
            \item \label{AS-retro}
            \quad $F$ is retrotone in $\Lambda$;
        \end{enumerate}
    \item \label{AS-3-strong}
    \quad for any $x, y \in C_{+}$, if $x < y$ then $f_i(x) > f_i(y)$ for all  $i$.
\end{enumerate}

Under \ref{AS4} we will occasionally need a modified form of~\ref{AS-3}, namely
\begin{enumerate}[label=\textup{\textoverline{A\arabic*}},ref=\textup{\textoverline{A\arabic*}}]
\setcounter{enumi}{3}
    \item \label{AS3-mod}
    \quad for any $x, y \in \Lambda$, if $F(x) < F(y)$ then
    \begin{enumerate}[label=\textup{\textoverline{A\arabic{enumi}}-\alph{enumii}},ref=\textup{\textoverline{A\arabic{enumi}}-\alph{enumii}}]
        \item \label{AS3-a-mod}
        \quad $f_i(x) \ge f_i(y)$ for all  $i$, and
        \item \label{AS3-b-mod}
        \quad for those $i$ for which $F_i(x) < F_i(y)$ there holds either $x_i = 0$ or $f_i(x) > f_i(y)$.
     \end{enumerate}
\end{enumerate}
Similarly, under \ref{AS4-strong} we will occasionally need a modified form of~\ref{AS-3-strong}, namely
\begin{enumerate}[label=\textup{\textoverline{A\arabic*}$'$},ref=\textoverline{A\arabic*}$'$]
\setcounter{enumi}{3}
    \item \label{AS3-strong-mod}
    \quad for any $x, y \in \Lambda$, if $F(x) < F(y)$ then $f_i(x) > f_i(y)$ for all  $i$, provided $x_i > 0$.
\end{enumerate}

As will be seen later, the assumptions \ref{AS-3} and \ref{AS4} are not, in~general, independent of each other. Our motivation is that we wish to strike a balance between assumptions that are reasonably general and, on the other hand, easy to check.

In~particular, since \ref{AS-3} and~\ref{AS-weak-retro} imply \ref{AS3-mod}, one may well ask why we have not chosen to assume the latter only.  The reason is that in the case when $F$ is given by a closed-form formula and is not~necessarily injective on the whole of $C_{+}$ (as, for~instance, in the planar Ricker model, see subsection~\ref{subsect:2DRicker}), the checking of whether \ref{AS3-mod} is satisfied could be a difficult task, whereas \ref{AS-3} is a simple consequence of the negativity of the relevant derivatives.  On the other hand, when $F$ is the time\nobreakdash-\hspace{0pt}one map in the semiflow generated by a competitive system of ODEs, both \ref{AS3-mod} and~\ref{AS-weak-retro} are fairly direct consequences of the Müller--Kamke theorem (see subsection~\ref{subsect:ODEs}), whereas we see no reason why \ref{AS-3} should be satisfied.

\medskip

The remainder of this section is devoted to the proof of the following existence theorem for the weak carrying simplex or carrying simplex:

\begin{theorem}
\label{prop1_steve}
    Under the assumptions \ref{AS-C^0}--\ref{AS4}, and \ref{AS-3} or \ref{AS3-mod} where $\varkappa$ can be arbitrary \textup{(}so that $\Lambda$ could be all of $C_+$\textup{)}, there exists a weak carrying simplex $\Sigma$.  If we assume additionally \ref{AS4-strong} or \ref{AS-3-strong}, $\Sigma$ is a carrying simplex.
    Moreover, $\Sigma$ satisfies the additional properties \ref{P-order-convex}--\ref{P-Lipschitz}.
\end{theorem}

\subsection{Restriction of the dynamical system $(F^n)_{n = 0}^{\infty}$ to the axes}
\label{subsect:axes}
For $F$ satisfying \ref{AS-C^0} we define one-dimensional  maps $G_i \colon [0, \infty) \to [0, \infty)$, $i \in \{1, \ldots,d\}$, through $G_i(s) := s g_i(s)$, where $g_i(s) := f_i(s \ev_i)$.  The map $G_i$ is  the dynamical rule $F$ restricted to the forward invariant $i$\nobreakdash-\hspace{0pt}th axis, and
$G^n_i$, for $n \in \NN$, denotes the $n$\nobreakdash-\hspace{0pt}th iterate of $G_i$.

In the remainder of the present subsection, the terms from the theory of dynamical systems, such as attractor, $\omega$\nobreakdash-\hspace{0pt}limit set, etc. will refer to each one-dimensional dynamical system $(G_i^n)_{n = 0}^{\infty}$, $i=1,\ldots,d$.
The following results are straightforward  to prove, cf., e.g., \cite[Lem.~6.6]{RH}.
\begin{lemma}
\label{lm:axes}
    Under \ref{AS-C^0}--\ref{AS-3}, for $i \in \{1, \ldots,d\}$ the following holds.
    \begin{enumerate}
        \item[\textup{(a)}]
        $1$ is the unique fixed point of $G_i$ on $(0, \infty)$.
        \item[\textup{(b)}]
         \begin{equation*}
        G_i(s)
        \begin{cases}
        \in (s, 1) & \text{for } s \in (0, 1),
        \\
        = s & \text{for } s = 1
        \\
        \in (0, s) & \text{for } s \in (1, \infty).
        \end{cases}
        \end{equation*}
        \item[\textup{(c)}]
        For any $s \in (0, \infty)$ the sequence $(G_i^n(s))$ converges, as $n \to \infty$, to $1$ in an eventually monotone way.
        \item[\textup{(d)}]
        For $s \in (1, 1 + \varkappa]$ there holds $G_i(s) \in (1, s)$, hence the sequence $(G_i^n(s))$ strictly decreases to $1$.
        \end{enumerate}

        \medskip
        Under \ref{AS-C^0}--\ref{AS4} and \ref{AS3-mod}, for $i \in \{1, \ldots,d\}$ the following holds.
        \begin{enumerate}
        \item[\textup{(a)$'$}]
        $1$ is the unique fixed point of $G_i$ on $(0, 1 + \varkappa]$.
        \item[\textup{(b)$'$}]
         \begin{equation*}
        G_i(s)
        \begin{cases}
        \in (s, 1) & \text{for } s \in (0, 1),
        \\
        = s & \text{for } s = 1
        \\
        \in (1, s) & \text{for } s \in (1, 1 + \varkappa].
        \end{cases}
        \end{equation*}
        \item[\textup{(c)$'$}]
        For any $s \in (0, 1)$ the sequence $(G_i^n(s))$ strictly increases, as $n \to \infty$, to $1$, and for any $s \in (1, 1 + \varkappa]$ the sequence $(G_i^n(s))$ strictly decreases, as $n \to \infty$, to $1$
        \end{enumerate}
\end{lemma}

\begin{lemma}
\label{lm:G_i-image}
    Let \ref{AS-C^0}--\ref{AS4} hold. Assume moreover \ref{AS-3} or \ref{AS3-mod}.  Let $a \in [0, 1 + \varkappa]$.  Then for each $i \in \{1, \ldots,d\}$ and each $n \in \NN$, $G_i^n$ is an increasing homeomorphism of $[0, a]$ onto $[0, G_i^n(a)]$.
\end{lemma}

\begin{proposition}
\label{prop:G_i}
    Let \ref{AS-C^0}--\ref{AS4} hold. Assume moreover \ref{AS-3} or \ref{AS3-mod} where $\varkappa$ can be arbitrary.  Then for each $i \in \{1, \ldots,d\}$, the invariant set $[0, 1]$ is the compact attractor of bounded sets in $[0, \infty)$.
\end{proposition}

\subsection{Existence of the compact attractor $\Gamma$ of bounded sets}
\label{subsect:Lambda-attracts}
Throughout the present subsection we assume \ref{AS-C^0}--\ref{AS4}.  Later on, our assumptions will be successively strengthened.

\begin{lemma}~
\label{lm:aux0}
    \begin{enumerate}
        \item[\textup{(a)}]
        Assume \ref{AS-3} or \ref{AS3-mod}.  Let $a \in (0, 1 + \varkappa]$.  Then $F^n\bigl([0,a \ev]\bigr) \subset [0, G^n(a)] = G^n([0, a])$ for all $n \in \NN$.
        \item[\textup{(b)}]
        Assume~\ref{AS-3}. Let $a > 0$.  Then $F^n\bigl([0, a \ev]\bigr) \subset  G^n([0, a])$ for all $n \in \NN$.
    \end{enumerate}
\end{lemma}
\begin{proof}
  We prove the lemma by induction on $n$.  For $n = 0$ we have $[0, a] = G^0_i([0, a])$.  Assume that the inclusion holds for some $n \in \NN$.

  In case (a), suppose to the contrary that there is  $x\in [0, a \ev]$ such that $F^{n+1}(x) \notin  [0, G^{n+1}(a)]$, which means that there are $j \in \{1, \ldots, d\}$ such that $F_j^{n+1}(x) > G^{n+1}_j(a)$.  Fix such a $j$.  We have thus
  \begin{equation*}
      F(F^{n}(x)) = F^{n+1}(x) > G^{n+1}_j(a) \ev_j = F(G^{n}_j(a) \ev_j)
  \end{equation*}
  with $F^{n}(x), G^{n}_j(a) \ev_j \in \Lambda$, so, by weak retrotonicity (\ref{AS-weak-retro}), $F_j^{n}(x) > G^{n}_j(a)$, which contradicts our inductive assumption.  The last equality is a consequence of Lemma~\ref{lm:G_i-image}.

  In case (b), for $x \in  [0,a \ev]$,
    \begin{equation*}
     \begin{aligned}
        F_i^{n+1}(x) = {} & F_i^n(x) f_{i}(F^n(x)) &
        \\
        \le {} & F_i^n(x) f_{i}(F_i^n(x) \ev_i)  & \text{ (by \ref{AS-3-competitive})}
        \\
        = {} & F^{n}_i(x) g_{i}(F^{n}_i(x)) = G_{i}(F^{n}_i(x)) & \text{ (by the definitions of $g_i$ and $G_i$)}
        \\
        \in {} & G_{i}(G^n_i([0, a])) & \text{ (by inductive hypothesis)}.
        \end{aligned}
  \end{equation*}
\end{proof}

From now on until the end of the present subsection we assume additionally \ref{AS-3} or \ref{AS3-mod}.
\begin{lemma}~
\label{lem3}
\begin{enumerate}
    \item[\textup{(a)}]
    For $n \in \NN$, $F^n(\Lambda) \subset [0,G^n(1+\varkappa)]$.  In~particular,    $F(\Lambda) \subset \inte_{\, C_{+}}{\Lambda}$.
    \item[\textup{(b)}]
    $F( [0, \ev]) \subset [0, \ev]$. Under \ref{AS4-strong} or \ref{AS-3-strong}, if $x \in [0,\ev] \setminus \{0, \ev_1, \dots, \ev_d\}$ then $F(x) \in \inte_{(C_{I(x)})_{+}}\bigl( [0, \ev]\bigr)$.
\end{enumerate}
\end{lemma}
\begin{proof}
The first sentences in (a) and (b) are direct consequences of Lemma~\ref{lm:aux0}.  The second sentence in (a) follows since, by Lemma~\ref{lm:axes}(b) or (b)$'$, $G_i(1 + \varkappa) < 1 + \varkappa$ for all $i$.

Assume \ref{AS4-strong}, and suppose to the contrary that there is $x \in  [0, \ev] \setminus \{0, \ev_1, \dots, \ev_d\}$ such that $F(x) \in \bd_{\, (C_{I(x)})_{+}}\bigl( [0, \ev]\bigr)$. Take $i \in I(x)$ such that $F_i(x) = 1$.  We have $F(x) > \ev_i$, and there is $j \in I(x)$, $j \ne i$, such that $x_j > 0$.  By retrotonicity (\ref{AS-retro}), $x \gg_{I(x)} \ev_i$, a contradiction.

Assume \ref{AS-3-strong}, and let $x \in [0, \ev] \setminus \{0, \ev_1, \dots, \ev_d\}$.  If $i \in I(x)$ is such that $x_{i} < 1$ then, by \ref{AS-3-strong} and Lemma~\ref{lm:axes}(b),
\begin{equation*}
    F_i(x) = x_{i} f_{i}(x) \le x_{i} g_{i}(x_{i} \ev_{i}) = G_{i}(x_{i}) < 1.
\end{equation*}
If $x_i = 1$ then, as $x \ne \ev_k$, $k = 1, \dots, d$, there is some $j \in I(x) \setminus \{i \}$ such that $x_j > 0$, so we can apply \ref{AS-3-strong} to conclude that
\begin{equation*}
    F_i(x) = x_{i} f_{i}(x) < x_{i} g_{i}(x_{i} \ev_{i}) = G_{i}(x_{i}) \le 1.
\end{equation*}
\end{proof}

Let us now define what will be shown to be the global attractor of compact sets, named $\Gamma$ in anticipation:
\begin{equation*}
  \Gamma := \bigcap\limits_{n=0}^{\infty} F^n(\Lambda).
\end{equation*}
Such a definition was given in \cite[Lem.~5.4]{Hou21}. $\Gamma$, being the intersection of a decreasing family of compact nonempty sets, is compact and nonempty.  Further, since $\Lambda \supset \ldots \supset F^{n}(\Lambda) \supset F^{n+1}(\Lambda) \ldots\ $, there holds $\Gamma = \omega(\Lambda)$.

\begin{lemma}
   $\Gamma = \bigcap\limits_{n=0}^{\infty} F^n\bigl([0, \ev]\bigr)$.  Moreover, $\ev_i \in \Gamma$ for all $i \in \{1, \ldots, d\}$.
\end{lemma}
\begin{proof}
  Since $[0, \ev] \subset \Lambda$, the `$\supset$' inclusion is straightforward.  To prove the other inclusion, observe first that it follows from Lemmas~\ref{lem3}(a) and~\ref{lm:axes}(c) or (c)$'$ that $\Gamma \subset  [0, \ev]$.  Consequently,
  \begin{equation*}
     \Gamma = \omega(\Gamma) \subset \omega\bigl([0, \ev]\bigr) = \bigcap\limits_{n=0}^{\infty} F^n\bigl([0, \ev]\bigr),
  \end{equation*}
  where the last equality holds since $[0, \ev] \supset \ldots \supset F^{n}\bigl( [0, \ev]\bigr) \supset F^{n+1}\bigl([0, \ev]\bigr) \ldots\ $. Finally, each $\ev_i \in \Gamma$ because they are fixed points contained in $\Lambda$.
\end{proof}

It should be remarked that in~\cite[Prop.~3.5]{SmithJDE} an analogue of $\Gamma$ was defined as this same set $\bigcap\limits_{n=0}^{\infty} F^n\bigl( [0, \ev]\bigr)$.

\medskip
Until the end of the present subsection, in case of \ref{AS3-mod} we assume furthermore that any positive number can serve as $\varkappa$.

\begin{lemma} \label{lm:dissipative-1}
For a bounded $A \subset C_{+}$ there is $n_0$ such that $F^n(A) \subset \Lambda$ for all $n \ge n_0$.  In~particular, $\Lambda$ attracts bounded sets $A \subset C_{+}$.
\end{lemma}
\begin{proof}
   Let $A \subset C_+$ be a bounded set and choose $a > 1$ such that $A \subset  [0,a \ev]$.   Since, by Proposition~\ref{prop:G_i}, for each $i \in \{1, \ldots, d\}$ in the dynamical system $(G_i^n)$ the set $[0, 1]$ attracts $[0, a]$, there exists $n_0 \in \NN$ such that for all $n \ge n_0$ and all $i \in \{1, \ldots, d\}$ the set $G^n_i([0, a])$ is contained in $[0, 1 + \varkappa)$.  From Lemma~\ref{lm:aux0} it follows that
   \begin{equation*}
      F^n\bigl( [0,a \ev]\bigr) \subset  G^n([0, a]) \subset \Lambda
   \end{equation*}
   for all $n \ge n_0$.
 \end{proof}
Collecting the various results of this subsection together we obtain the following characterisation of $\Gamma$:
\begin{theorem}~
\label{thm:global-attractor-exists}
    \begin{enumerate}
        \item[\textnormal{(a)}]
        $\Gamma$ is the compact attractor of bounded sets in $C_{+}$.
        \item[\textnormal{(b)}]
        $\Gamma$ is characterised as the set of those $x \in C_{+}$ for which there exists a bounded total orbit.
    \end{enumerate}
\end{theorem}
\begin{proof}
  Let $A \subset C_{+}$ be bounded.  By Lemma~\ref{lm:dissipative-1}, $(F^n)$ is asymptotically compact on $A$; moreover, $\omega(A) \subset \omega(\Lambda) = \Gamma$.  Therefore, by Lemma~\ref{lm:attracting-equivalent}, $\Gamma$ attracts $A$.  This proves (a). The characterisation given in (b) is a consequence of Proposition~\ref{prop:attractor-characterization}.
\end{proof}

A consequence of Theorem~\ref{thm:global-attractor-exists} is
\begin{lemma}
\label{lm:attractor-characterization}
  For $x \in C_{+}$ the following are equivalent.
  \begin{enumerate}
    \item[\textup{(1)}]
    $x \in \Gamma$.
    \item[\textup{(2)}]
    There is a total orbit of $x$, contained in $\Lambda$.
    \item[\textup{(3)}]
    There is a bounded backward orbit of $x$.
  \end{enumerate}
\end{lemma}
$\Gamma$ is the largest compact invariant set in $C_{+}$ (see~\cite[Thm.~2.19 on p.~37]{S-T}).  Further, since $\Gamma \subset \Lambda$ and, by Lemma~\ref{lm:AS-homeo}, $F{\restriction}_{\Lambda}$ is a homeomorphism onto its image, $F{\restriction}_{\Gamma}$ is a homeomorphism onto $\Gamma$, so $((F{\restriction}_{\Gamma})^n)_{-\infty}^{\infty}$ is a (two-sided) dynamical system on the compact metric space $\Gamma$ (and $\Gamma$ is the largest subset of $C_{+}$ with that property).

(In some papers (\cite{HirschJBD, JandW}) $\Gamma$ is called the \textit{global attractor} for $F$.)

\subsection{Construction of the carrying simplex}
\label{subsect:construction-of-CS}
In the present subsection we always assume \ref{AS-C^0}--\ref{AS4}, and \ref{AS-3} or \ref{AS3-mod}.  At some places we assume \ref{AS4-strong} or \ref{AS-3-strong}. For convenience we write $F$ instead of $F{\restriction}_{\Lambda}$.

Denote by $T$ the radial projection of $C_{+} \setminus \{0\}$ onto the unit probability simplex  $\Delta$, $T(x) := x/\normone{x}$.

Following \cite{baigent_comp}, let $\widehat{\mathcal{U}}$ [resp.\ $\cU$] stand for the set of bounded and weakly unordered [resp.\ unordered] hypersurfaces contained in $\Lambda$ that are homeomorphic to the standard probability simplex $\Delta$ via radial projection.  In~particular, a hypersurface $S \in \widehat{\mathcal{U}}$ is at a positive distance from the origin, as the radial projection is not defined at $0$.

It follows that for any $S \in \widehat{\mathcal{U}}$ the inverse of the restriction  $(T{\restriction}_{S})^{-1}$ can be written as
\begin{equation*}
    \Delta \ni u \mapsto R(u) u,
\end{equation*}
where $R \colon \Delta \to (0, \infty)$ is a continuous function (called the \textit{radial representation} of $S$).  In other words,
\begin{equation*}
    x = R(T(x)) T(x), \quad x \in S.
\end{equation*}

For $S \in \widehat{\mathcal{U}}$, its complement $C_+ \setminus S$ is the union of two disjoint sets, a bounded one, $S_{-}$, and an unbounded one, $S_{+}$.  One has $S_{-} = [0, 1) S$, $S_{+} = (1, \infty) S$. We will write $\partial S$ for $S \cap \partial C_{+}$.  $\partial S = \bd_{C_{+}}S$.

\medskip

The following consequence of the weak unorderedness of an element of $\widehat{\mathcal{U}}$ will be used several times, so we formulate it as a separate lemma.
\begin{lemma}
\label{lm:aux-U}
    Let $S \in \widehat{\mathcal{U}}$.  There are no two points $x < y$ on $S$ such that $x_i < y_i$ for all $i \in I(y)$.
\end{lemma}
\begin{proof}
  Suppose to the contrary that there are such $x$ and $y$.  Then $x, y \in (C_{I(y)})_{+}$ with $x \ll_{I(y)} y$, which contradicts the fact that $S$ is weakly unordered.
\end{proof}

\begin{lemma} \label{cUintocU}
If $S \in \widehat{\mathcal{U}}$ then $F(S) \in \widehat{\mathcal{U}}$.  If $S \in \mathcal{U}$ then $F(S) \in \mathcal{U}$.  If \ref{AS4-strong} or \ref{AS-3-strong} holds, $F$ maps $\widehat{\mathcal{U}}$ into $\mathcal{U}$.
\end{lemma}
\begin{proof}
  Let $S \in \widehat{\mathcal{U}}$.  We prove first that $F(S)$ is weakly unordered.  Indeed, suppose there are $\emptyset \ne I \subset \{1, \ldots, d\}$ and $x, y \in S \cap (C_I)_{++}$ such that $F(x) \ll_I F(y)$.  Then, by weak retrotonicity (\ref{AS-weak-retro}), $x \ll_{I} y$, which is impossible.

  Since $F$ is continuous on $C_+$ and $S  \subset \Lambda$ is compact, $F(S)$ is compact, and, by Lemma~\ref{lem3}(a), $F(S)$ is a subset of $\Lambda'$.  Define $\hat{F}$ on $\Lambda'$ as $\hat{F} = F/\normone{F}$, in other words, $\hat{F} = T \circ F$. Then $\hat{F}$ is a continuous map on $\Lambda'$, and consequently the continuous map $\hat{F}{\restriction}_{S} \colon S \to \Delta$ is proper.  By Lemma~\ref{lm:AS-homeo}, $F$ is invertible, and $T{\restriction}_{F(S)}$ is locally invertible at each element of $F(S)$, since otherwise $F(S)$ would not be weakly unordered, which has been excluded in the previous paragraph.  Hence $\hat{F}{\restriction}_{S}$ is locally invertible at each $x \in S$, and we can apply \cite[Cor.~2.3.6]{ChHa} to conclude that $\hat{F}{\restriction}_{S}$ is a homeomorphism onto $\Delta$.  Therefore, $F(S) \in \widehat{\mathcal{U}}$.

  Let $S \in \mathcal{U}$.  If there are $x, y \in S$ such that $F(x) < F(y)$ then, by weak retrotonicity (\ref{AS-weak-retro}), $x < y$, which contradicts the unorderedness of $S$.

  Assume that \ref{AS4-strong} or \ref{AS-3-strong} holds, and let $S \in \widehat{\mathcal{U}}$. Suppose to the contrary that $F(S)$ is not unordered, that is, there are $x, y \in S$ such that $F(x) < F(y)$.

 In the case of \ref{AS4-strong} it follows from retrotonicity (\ref{AS-retro}) that $x_i < y_i$ for all $i \in I(y) = I(F(y))$, which is in contradiction to Lemma~\ref{lm:aux-U}.

  We consider now the case of~\ref{AS-3-strong}.  We already know that $F(S)$ is weakly unordered, so, as a consequence of Lemma~\ref{lm:aux-U}, $I(F(y))$ is the disjoint union of two nonempty sets, $J := \{\, i \in I(F(y)) : 0 < F_i(x) = F_i(y) \,\}$ and $K := \{\, i \in I(F(y)) : F_i(x) < F_i(y) \,\}$.  By weak retrotonicity (\ref{AS-weak-retro}), $x_i \le y_i$ for all $i \in I(y) = I(F(y))$, with $x_i <  y_i$ for all $i \in K$.  Applying again Lemma~\ref{lm:aux-U}, this time to $S$, we obtain that $x_j = y_j$ for at~least one $j \in J$.  Fix such a $j$.  But \ref{AS-3-strong} gives us $f_j(x) > f_j(y)$, hence $F_j(x) = x_j f_j(x) > y_j f_j(y) = F_j(y)$, which contradicts the fact that $j \in J$.
\end{proof}

\medskip

We now introduce a partial order relation for the hypersurfaces in $\widehat{\mathcal{U}}$.  For $S, S' \in \widehat{\mathcal{U}}$, let $R, R'$ denote their respective radial representations.  We write $S \preccurlyeq S'$ if $R(u) \le R'(u)$ for all $u \in \Delta$, $S \prec S'$ if $S \preccurlyeq S'$ and $S \ne S'$, and $S \llcurly S'$ if $R(u) < R'(u)$ for all $u \in \Delta$.  It is straightforward that $S \prec S'$ if and only if $R(u) \le R'(u)$ for all $u \in \Delta$ and there is $v \in \Delta$ with $R(v) < R'(v)$ and observe that $S \preccurlyeq S'$ if and only if $S \subset S' \cup (S')_{-}$ (or, which is the same, $S' \subset S \cup S_{+}$).

We assume the convention that $0 \llcurly S$ for any $S \in \widehat{\mathcal{U}}$.

For $S \preccurlyeq S'$ we denote
\begin{equation*}
    \langle S, S' \rangle := [1, \infty) S \cap [0, 1] S',
\end{equation*}
and for $S \llcurly S'$ we denote
\begin{equation*}
    \llangle S, S' \rrangle := (1, \infty) S \cap [0, 1) S'.
\end{equation*}
In other words, $\langle S, S' \rangle = \{\,\lambda u: \lambda \in [R(u), R'(u)], u \in \Delta\,\}$, and  $\llangle S, S' \rrangle = \{\,\lambda u: \lambda \in (R(u), R'(u)), u \in \Delta\,\}$.  In~particular, $\langle 0, S \rangle = \{\,\lambda u: 0 \le \lambda \le R(u), u \in \Delta\,\} = [0, 1] S$.

Notice that $S \preccurlyeq S'$ if and only if $S \subset \langle 0, S' \rangle$.

\medskip

The following lemma shows that the volume $\langle S, S' \rangle$ between two hypersurfaces in $\widehat{\mathcal{U}}$ is the union of order intervals and hence is order convex.
\begin{lemma}
\label{lm:between}
    For $S, S' \in \widehat{\mathcal{U}}$ with $S \preccurlyeq S'$,
    \begin{equation*}
        \langle S, S' \rangle = \bigcup_{\substack{x \in S \\ y \in S'}} [x, y]
    \end{equation*}
  and  is order convex.
\end{lemma}
\begin{proof}
  The `$\subset$' inclusion is straightforward.  Assume that $z \in [x, y]$ with $x \in S$ and $y \in S'$.  There are $\xi \in S$ and $\eta \in S'$ such that $z = \alpha \xi = \beta \eta$.  We claim that $\alpha \ge 1$ and $\beta \le 1$.  Indeed, suppose that $\alpha < 1$.  Then $x \le z \ll_{I} \xi$, where $I := I(z) = I(\xi)$, with both $x$ and $\xi$ in $S$.  If $I(x) = I$ then $x \ll_{I} \xi$, which contradicts the weak unorderedness of $S$.  If not, we can find $\tilde{x} \in S \cap (C_{++})_I$ so close to $x$ that $\tilde{x} \ll_{I} \xi$, which again contradicts the weak unorderedness of $S$.  The case $\beta > 1$ is excluded in much the same way. We have thus $z \in [R(T(z)), R'(T(z))] T(z)$, with $\xi = R(T(z)) T(z)$ and $\eta = R'(T(z)) T(z)$. Finally, as a consequence of the equality, $\langle S, S' \rangle$ is order convex.
\end{proof}

The property described in the result below has appeared in the literature, see, e.g., \cite[Prop.~2.1]{J-M-W}.  As the assumptions of $F$ made in the present paper are different (e.g. weak retrotonicity), we have decided to give its reasonably complete proof.

\begin{lemma}
   \label{lm:AS4-conseq}
  Assume~\ref{AS-C^0}--\ref{AS4}.  Then for any $x, y \in \Lambda$, if $F(x) \leq F(y)$, then
  \begin{enumerate}
      \item[\textup{(i)}]
      $x \leq y$,
      \item[\textup{(ii)}]
      $[0, F(y)] \subset F([0, y])$,
      \item[\textup{(iii)}]
      $[F(x),F(y)] \subset F([x,y])$.
  \end{enumerate}
\end{lemma}
\begin{proof}
Suppose that $F(x)\leq F(y)$. If $F(x) = F(y)$, then as, by Lemma~\ref{lm:AS-homeo}, $F$ is a homeomorphism of $\Lambda$ onto its image, $x = y$. This leaves the case $F(x) < F(y)$, when the statement (i) follows from \ref{AS-weak-retro}.

Regarding (ii), the conclusion is obvious if $y = 0$.  Assume $y > 0$ with support $I = I(y)$.  Then, by Lemma~\ref{cUintocU} for $\widehat{\mathcal{U}}$ restricted to $(C_I)_{+}$,
\begin{equation*}
    S = \bd_{(C_{I})_{+}}[0, y] \in \widehat{\mathcal{U}}\!\!\restriction_{(C_{I})_{+}} \quad \text{and} \quad F(S) = F(\bd_{(C_{I})_{+}}[0, y]) \in \widehat{\mathcal{U}}\!\!\restriction_{(C_{I})_{+}}.
\end{equation*}
$F\!\!\restriction_{\Lambda}$, being a homeomorphism of $\Lambda$ onto $F(\Lambda)$, takes $\bd_{\Lambda_I} [0, y] = \bd_{(C_I)_{+}} [0, y] = S$ onto $\bd_{\Lambda_I} F([0, y]) = \bd_{(C_I)_{+}} F([0, y]) = F(S)$.  So $F([0, y])$ is a compact set contained in $(C_I)_{+}$ whose relative boundary in $(C_I)_{+}$ equals $F(S)$.  Further, $[0, y]$ is a $(\card{I})$\nobreakdash-\hspace{0pt}dimensional topological disk whose manifold boundary, that is $([0, y] \cap \partial (C_I)_{+}) \cup S$, separates $\RR^d_{I}$ into two components, one of them (the bounded one) being just $[0, y]$ (the Jordan--Brouwer separation theorem, see, e.g. \cite[Thm.~4.7.15 on p.~198]{Spanier}).  The restriction $F\!\!\restriction_{[0, y]}$, being a homeomorphism onto its image, preserves the manifold boundary.  It follows from~\ref{AS-C^0} that $F(([0, y] \cap \partial (C_I)_{+}) \cup S) = (F([0, y]) \cap \partial (C_I)_{+}) \cup F(S)$.  The manifold boundary of $F([0, y])$ separates $\RR^d_{I}$ into two components, the bounded one being just $F([0, y])$.  But that component is just equal to $\langle 0, F(S) \rangle$.  For more on manifolds, their boundaries, etc., see \cite[pp.~24--27]{Lee}.  Since $F(y) \in F(S)$, Lemma~\ref{lm:between} shows $[0,F(y)] \subset \langle 0, F(S) \rangle = F([0,y])$, as required.

    Now we show that for $x,y\in \Lambda$ such that $F(x) \leq F(y)$ we have $[F(x),F(y)] \subset F([x,y])$. If $F(x) = F(y)$ then as $F$ is a homeomorphism $x = y$ and the inequality holds trivially. Thus suppose $F(x) < F(y)$. Let $\zeta \in [F(x), F(y)]$ so that
    $F(x) \leq \zeta \leq F(y)$. In particular $0 \leq \zeta \leq F(y)$ so that $\zeta \in [0,F(y)]$. By the previous paragraph, $\zeta \in [0,F(y)] \subset F([0,y])$. Thus there exists $z \in [0,y]$ such that $F(z) = \zeta$, and then $F(x)\leq F(z) \le F(y)$, which by weak retrotonicity gives $x \leq z \leq y$, i.e. $z \in [x,y]$.
\end{proof}

\begin{lemma}
\label{lm:aux3}
    For $S \in \widehat{\mathcal{U}}$, $F(\langle 0, S \rangle) = \langle 0, F(S) \rangle$.
\end{lemma}
\begin{proof}
  \begin{equation*}
       \begin{aligned}
        \langle 0, F(S) \rangle & = \bigcup_{y \in F(S)} [0, y] & \text{ (by Lemma~\ref{lm:between})}
        \\
        & = \bigcup_{x \in S} [0, F(x)]
        \\
        & \subset \bigcup_{x \in S} F([0, x]) & \text{ (by Lemma~\ref{lm:AS4-conseq})}
        \\
        & = F\Bigl(\bigcup_{x \in S} [0, x]\Bigr)
        \\
        & = F(\langle 0, S \rangle) & \text{ (by Lemma~\ref{lm:between})}.
      \end{aligned}
  \end{equation*}
   Suppose to the contrary that there is $z \in F(\langle 0, S \rangle) \setminus \langle 0, F(S) \rangle$.  Let $y$ be the unique member of $F(S)$ such that $y = {\alpha} z$ with $\alpha > 0$.  By our choice of $z$, we have $\alpha < 1$ and $y \ll_{I} z$, where $I := I(z) = I(y)$.  Let $\eta \in S$ and $\zeta \in \langle 0, S \rangle$ be such that $F(\eta) = y$ and $F(\zeta) = z$.  Weak retrotonicity (\ref{AS-weak-retro}) yields $\eta \ll_{I} \zeta$.  Let $x$ be the unique member of $S$ such that $x = {\beta} \zeta$ with $\beta \ge 1$.  Then $x \in S \cap (C_I)_{++}$, and either $\zeta = x$ or $\zeta \ll_{I} x$.  At any rate, $\eta \ll_I x$, which contradicts the weak unorderedness of $S$.
\end{proof}

\begin{lemma}
\label{lm:order-preserving-U}
   $F$ preserves the $\preccurlyeq$ and $\llcurly$ relations on $\widehat{\mathcal{U}}$.
\end{lemma}
\begin{proof}
  Let $S, S' \in \widehat{\mathcal{U}}$, $S \preccurlyeq S'$, which means that $S \subset \langle 0, S' \rangle$.  Then, by Lemma~\ref{lm:aux3}, $F(S) \subset F(\langle 0, S' \rangle) = \langle 0, F(S') \rangle$, that is, $F(S) \preccurlyeq F(S')$.

  Assume that $S \llcurly S'$.  By the previous paragraph, $F(S) \preccurlyeq F(S')$, and by the fact that $F{\restriction}_{\Lambda}$ is a homeomorphism onto its image, $F(S)$ and $F(S')$ are disjoint, consequently $F(S) \llcurly F(S')$.
\end{proof}

\medskip
For any $i \in \{1, \ldots, d\,\}$, by \ref{AS-e}, $f_i(\ev_i) = 1$, and $0 = F(0) < F(\ev_i) = \ev_i$ with $0_i < (\ev_i)_i$, so, by \ref{AS-3} or \ref{AS3-mod}, $f_i(0) > 1$.  Let $\varepsilon \in (0, 1)$ be small enough that $\min\limits_{1 \le i \le d}f_i(x) > 1$ for all $x \in \langle 0, S_0 \rangle$, where $S_0 :=\varepsilon \Delta$.  $S_0 \subset \Lambda$, is unordered and projects homeomorphically on $\Delta$, so it belongs to $\mathcal{U}$.

\begin{proposition}
\label{prop:0-repeller}
    $\{0\}$ is a uniform repeller for $(F^n)_{n = 0}^{\infty}$.
\end{proposition}

\begin{proof}
  (Recall that  we are not assuming that $F$ is $C^1$). As $U$ in the definition of uniform repeller we take $\langle 0, S_0 \rangle$.  It follows from the choice of $S_0$ that there is $\delta > 0$ such that $f_i(x) \ge 1 + \delta$, $i \in \{1, \ldots, d\}$, consequently $\normone{F(x)} \ge (1 + \delta) \normone{x}$, for all $x \in U$.   So, if $x \in U \setminus \{0\}$ we have that there exists $n_0 \in \NN$ with the property that $F^{n_0}(x) \notin U$ ($n_0$ is not larger than the least nonnegative integer $n$ such that $(1 + \delta)^n \normone{x} \ge \varepsilon$, where $\varepsilon > 0$ is as in the definition of $S_0$).  Since $F{\restriction}_{\Lambda}$ is a homeomorphism onto its image contained in $\Lambda$, there holds $F^n(x) \notin U$ for all $n \ge n_0$.
\end{proof}

\smallskip
The sequence $S_n := F^n(S_0) \in \mathcal{U}$ by Lemma \ref{cUintocU}.  Denote by $R_n \colon \Delta \to (0, \infty)$ the radial representation of $S_n$.  By our choice of $S_0$, we have $S_0 \llcurly S_1$.  Lemma~\ref{lm:order-preserving-U} implies $S_n \llcurly S_{n + 1}$ for all $n = 1, 2, \dots$.  As a consequence, for each $u \in \Delta$ the sequence $(R_n(u))_{n = 0}^{\infty}$ is strictly increasing.  Let $R_*$ stand for its pointwise limit.

We define
\begin{equation*}
    S_{*} := \{\, R_{*}(u) u : u \in \Delta \,\}.
\end{equation*}
We recall the definition of Kuratowski limit.
\begin{definition}
\label{def:Kuratowski}
    $\widetilde{S}$ is the \emph{Kuratowski limit} of the sequence $(S_n)_{n = 0}^{\infty}$ if the following conditions are satisfied:
    \begin{enumerate}
    \item[\textup{(K1)}]
    For each $x \in \widetilde{S}$ there is a sequence $(x_{n})_{n = 0}^{\infty}$ such that $x_n \in S_n$ and $\lim\limits_{n \to \infty} x_n = x$.
    \item[\textup{(K2)}]
    For any sequence $(x_{n_k})_{k = 1}^{\infty}$ with $x_{n_k} \in S_{n_k}$ and $n_k \underset{k \to \infty}{\longrightarrow} \infty$, if $\lim\limits_{k \to \infty} x_{n_k} = x$ then $x \in \widetilde{S}$.
    \end{enumerate}
\end{definition}
See~\cite[Def.~4.4.13]{AT}.
Recall that $d_{H}$ stands for the Hausdorff metric on the family $2^{\Lambda}$ of nonempty compact subsets of $\Lambda$.  $(2^{\Lambda}, d_{H})$ is a compact metric space (see~\cite[Thm.~4.4.15]{AT}).
\begin{lemma}~
\label{lm:properties-S_*}
    \begin{enumerate}
        \item[\textup{(1)}]
        $S_{*} \subset \widetilde{S}$.
        \item[\textup{(2)}]
        $S_{*}$ is weakly unordered.
    \end{enumerate}
\end{lemma}
\begin{proof}
    (1) is a consequence of the definitions of $S_{*}$ and  $\widetilde{S}$.  Suppose to the contrary that $x, y \in S_* \cap (C_I)_{+}$ are such that $x \ll_I y$.  By construction, there are $u, v \in \Delta_I$ such that $x = \lim\limits_{n \to \infty} R_n(u) u$ and $y = \lim\limits_{n \to \infty} R_n(v) v$.  Since the relation $\ll_I$ is relatively open in $(C_I)_{+}$, $R_n(u) u \ll_{I} R_n(v) v$ for $n$ sufficiently large.  But $R_n(u) u, R_n(v) v \in S_{n} \cap (C_I)_{+}$, which contradicts the weak unorderedness of $S_n$.
\end{proof}

\begin{lemma}~
\label{lm:properties-tilde-S}
    \begin{enumerate}
        \item[\textup{(1)}]
        $d_{H}(S_n, \widetilde{S}) \to 0$  as $n \to \infty$.
        \item[\textup{(2)}]
        $\widetilde{S}$ is invariant under $F$.
    \end{enumerate}
\end{lemma}
Recall that $F$ denotes here $F{\restriction}_{\Lambda}$, so the invariance of $\widetilde{S}$ means that, first, if $x \in \widetilde{S}$ then $F(x) \in \widetilde{S}$, and, second, if $x \in \Lambda$ is such that $F(x) \in \widetilde{S}$ then $x \in \widetilde{S}$.

\begin{proof}
    Part (1) is a consequence of the fact that for the compact $\Lambda$, the Kuratowski convergence and the convergence in the Hausdorff metric are the same, see~\cite[Prop.~4.4.15]{AT}.

    As a consequence, $d_H(S_{n+1}, \widetilde{S}) \to 0$, as $n \to \infty$.  But $S_{n+1} = F(S_n)$, hence $d_H(F(S_{n}), \widetilde{S}) \to 0$ as $n \to \infty$.  On the other hand, $F(S_n)$ converge to $F(\widetilde{S})$ in the Hausdorff metric, so $\widetilde{S} = F(\widetilde{S})$.
\end{proof}

Next we consider the convergence of suitably generated decreasing sequences of weakly unordered sets.

\medskip

Put $S^{0} := \bd_{C_{+}} \Lambda = H(1 + \varkappa) \in \widehat{\mathcal{U}}$ (see Example~\ref{ex:weakly-unordered-not-ordered}).  We will consider a sequence $(S^{n})_{n = 0}^{\infty}$, where $S^{n} := F^{n}(S^0)$.
 Lemma~\ref{cUintocU} shows that  $S^n := F^n(S^0)\in \widehat{\mathcal{U}}$ for $n = 1, 2 , \dots$.

By Lemma~\ref{lem3}(a), we have $S^1 \llcurly S^0$ and by Lemma~\ref{lm:order-preserving-U}, $S^{n+1} \llcurly S^n$ for all $n = 1, 2, \dots$.  Denote by $R^n \colon \Delta \to (0, \infty)$ the radial representation of $S^n$.  Since $S_0 \llcurly S^0$, it follows from Lemma~\ref{lm:order-preserving-U} that $S_n \llcurly S^n \llcurly \ldots \llcurly S^1 \llcurly S^0$ for all $n \in \NN$, consequently $S_{*} \llcurly S^0$.  Again by Lemma~\ref{lm:order-preserving-U}, $S_{*} \llcurly S^n$ for all $n \in \mathbb{N}$, which means that $R_{*}(u) < \ldots < R^{n+1}(u) < R^{n}(u) < \ldots < R^0(u)$ for all $u \in \Delta$.  Therefore there exists $\lim\limits_{n \to \infty} R^n(u) =: R^{*}(u) \in [R_{*}(u), R^{0}(u))$. Define $S^{*} := \{\, R^{*}(u) u : u \in \Delta \,\}$, and put $\widecheck{S}$ to be the Kuratowski limit of the sequence $(S^{n})_{n = 0}^{\infty}$.  Note that by construction $\widecheck{S} \subset [1, \infty) S_{*}$.

The following are analogues of Lemmas~\ref{lm:properties-S_*} and~\ref{lm:properties-tilde-S}.
\begin{lemma}~
\label{lm:properties-S^*}
    \begin{enumerate}
        \item[\textup{(1)}]
        $S^{*} \subset \widecheck{S}$.
        \item[\textup{(2)}]
        $S^{*}$ is weakly unordered.
    \end{enumerate}
\end{lemma}
\begin{lemma}~
\label{lm:properties-check-S}
    \begin{enumerate}
        \item[\textup{(1)}]
        $d_{H}(S^n, \widecheck{S}) \to 0$  as $n \to \infty$.
        \item[\textup{(2)}]
        $\widecheck{S}$ is invariant under $F$.
    \end{enumerate}
\end{lemma}

We introduce now the Harnack metric on $C_{+}$.  For more details see~\cite{Kr}.

For $x, y \in C_{+}$, $x \ne 0$, we define
\begin{equation*}
    \lambda(x, y) := \sup\{\, \lambda \ge 0 : y - {\lambda}x \in C_{+} \, \}
    \\
    = \min\left\{ \frac{y_i}{x_i} : 1 \le i \le d \text{ such that } x_i > 0 \,\right\}.
\end{equation*}
We put $\lambda(x, y) = \infty$ if and only if $x = 0$.
Following~\cite{Kr}, we call $\lambda$ the \textit{order function} on $C_{+}$.  $\mu(x, y) := \min\{\lambda(x, y), \lambda(y, x)\}$ is the \textit{symmetrised order function} on $C_{+}$.

We define on $C_{+} \setminus \{0\}$ the \textit{Harnack metric}:
\begin{equation*}
    h(x, y) := 1 - \mu(x, y).
\end{equation*}

Let $x \le y$ both belong to $(C_I)_{++}$ for some $\emptyset \ne I \subset \{1, \dots, d\}$.  We have
\begin{equation*}
    \mu(x, y) = \lambda(y, x) = \min\left\{ \frac{x_i}{y_i} : i \in I \,\right\}.
\end{equation*}
The following is straightforward.
\begin{lemma}
\label{lm:1}
    Let $\lim\limits_{n \to \infty} x_n = x$, $\lim\limits_{n \to \infty} y_n = y$, $x_n \le y_n$ for all $n \in \NN$, and $x_n, y_n, x, y \in (C_I)_{++}$ for some $\emptyset \ne I \subset \{1, \dots, d\}$. Then $h(x, y) = \lim\limits_{n \to \infty} h(x_n, y_n)$.
\end{lemma}

\begin{lemma}
\label{lem:aux1}
  Let $\emptyset \ne I \subset \{1, \ldots, d\}$ and $x < y$, $x, y \in \Lambda \cap (C_I)_{++}$, be such that $F(x) < F(y)$ and $F_i(x) < F_i(y)$ for $i \in I$.  Then for each $\emptyset \ne J \subset I$ there holds
    \begin{equation*}
        \mu(\pi_J(F(x)), \pi_J(F(y))) > \mu(\pi_J(x), \pi_J(y)).
    \end{equation*}
    Consequently
    \begin{equation*}
        h(\pi_J(F(x)), \pi_J(F(y))) < h(\pi_J(x), \pi_J(y)).
    \end{equation*}
\end{lemma}
\begin{proof}
    We claim that
    \begin{equation*}
        \frac{F_i(x)}{F_i(y)} = \frac{x_i f_i(x)}{y_i f_i(y)} > \frac{x_i}{y_i}.
    \end{equation*}
    Indeed, the above follows directly from \ref{AS3-b-mod}.  In case of \ref{AS-3}, by weak retrotonicity in $\Lambda$ (\ref{AS-weak-retro}), for each $i \in I$ there holds $x_i < y_i$, and we can apply \ref{AS-3-competitive}.

    Therefore,
    \begin{align*}
       \mu(\pi_J(x), \pi_J(y)) & =  \min\left\{ \frac{x_i}{y_i} : i \in J \,\right\}
       \\[0.5ex]
       & < \min\left\{ \frac{F_i(x)}{F_i(y)} : i \in J \,\right\} = \mu(\pi_J(F(x)), \pi_J(F(y))).
    \end{align*}
\end{proof}

Now with the aid of the Harnack metric we may establish

\begin{theorem}
\label{thm:S_*=S^*}
    $S_{*} = \widetilde{S} = \widecheck{S} = S^{*}$.
\end{theorem}
\begin{proof}
    We start by proving that $\widetilde{S} = \widecheck{S}$. Suppose not.  This means, in view of Lemmas~\ref{lm:properties-S_*}(1) and~\ref{lm:properties-S^*}(1), that there is $v \in \Delta$ such that the half-line starting at $0$ and passing through $v$ intersects $\widetilde{S}$ at $x$ and intersects $\widecheck{S}$ at $y$, where $x < y$.  Put $I :=  I(x)$.  We have $x, y \in (C_I)_{++}$ and $x \ll_{I} y$.

    By Lemmas~\ref{lm:properties-tilde-S}(2) and~\ref{lm:properties-check-S}(2), $\widetilde{S}$ and $\widecheck{S}$ are invariant.

    From weak retrotonicity (\ref{AS-weak-retro}) it follows that $F^{-n}(x) \ll_{I} F^{-n}(y)$ for all $n \in \NN$.

    We write
    \begin{equation*}
        \alpha(x) := \{\, z \in \Lambda : \text{there is } n_k \to \infty \text{ such that } F^{-n_k}(x) \to z \text{ as } k \to \infty \,\},
    \end{equation*}
    and similarly for $\alpha(y)$.

    Fix $\xi \in \alpha(x)$.  By passing to a subsequence, if~necessary, we can assume that $F^{-n_k}(y)$ converges to $\eta \in \alpha(y)$ for the same sequence $n_k \to \infty$.  There holds, by the closedeness of the $\le$ relation, $\xi \le \eta$.  Put $\emptyset \ne J := I(\xi)$.  Then $\xi = \pi_J(\xi)$ and $\pi_J(\xi), \pi_J(\eta) \in (C_J)_{++}$.  We have, in view of Lemmas~\ref{lm:1} and \ref{lem:aux1},
    \begin{align*}
        h(\pi_J(\xi), \pi_J(\eta)) & = \lim_{k \to \infty} h(\pi_J(F^{-n_k}(x)), \pi_J(F^{-n_k}(y))
        \\
        & = \sup_{n \in \NN} h(\pi_J(F^{-n}(x)), \pi_J(F^{-n}(y)) > 0,
    \end{align*}
    from which it follows that $\pi_J(\xi) < \pi_J(\eta)$.  Let $K \ne \emptyset $ be the set of those $j \in J$ for which $\xi_j < \eta_j$.  Weak retrotonicity (\ref{AS-weak-retro}) implies $F^{-1}(\xi) < F^{-1}(\eta)$ and $(F^{-1})_j(\xi) <  (F^{-1})_j(\eta)$ for all $j \in K$.  We have, by Lemmas~\ref{lm:1} and \ref{lem:aux1},
    \begin{align*}
        h(\pi_K(\xi), \pi_K(\eta)) & = \lim_{k \to \infty} h(\pi_K(F^{-n_k}(x)), \pi_K(F^{-n_k}(y))
        \\
        & = \sup_{n \in \NN} h(\pi_{K}(F^{-n}(x)), \pi_{K}F^{-n}(y))
        \\
        & = \lim_{k \to \infty} h(\pi_K(F^{-n_k-1}(x)), \pi_K(F^{-n_k-1}(y))
        \\
        & =  h(\pi_K(F^{-1}(\xi)), \pi_K(F^{-1}(\eta))).
    \end{align*}
    But, again by Lemma~\ref{lem:aux1}, $h(\pi_K(F^{-1}(\xi)), \pi_K(F^{-1}(\eta))) > h(\pi_K(\xi), \pi_K(\eta))$, a contradiction.

    In a similar way we prove that it is impossible to have $x, y \in \widetilde{S}$ [resp.\ $x, y \in \widecheck{S}$] with $x < y$.  As a consequence we obtain the equality $S_{*} = \widetilde{S}$ [resp.\ $S^{*} = \widecheck{S}$].
\end{proof}

\begin{proposition}~
\label{prop:properties-limits}
    \begin{enumerate}
        \item[\textup{(a)}]
        The functions $R_n \colon \Delta \to (0, \infty)$ converge uniformly, as $n \to \infty$, to $R_{*}$.
        \item[\textup{(b)}]
        The functions $R^n \colon \Delta \to (0, \infty)$ converge uniformly, as $n \to \infty$, to $R_{*}$.
        \item[\textup{(c)}]
        $S_{*} \in \widehat{\mathcal{U}}$. Under \ref{AS4-strong} or \ref{AS-3-strong}, $S_{*} \in \mathcal{U}$.
        \item[\textup{(d)}]
        $S_{*}$ is invariant under $F$.
    \end{enumerate}
\end{proposition}
\begin{proof}
  Since $\Delta$ is compact, we use the fact that uniform convergence (to a function that is necessarily continuous) is equivalent to \textit{continuous convergence}:  for any sequence $(u_n)_{n = 1}^{\infty} \subset \Delta$ convergent to $u$ there holds $\lim\limits_{n \to \infty} R_n(u_n) = R_*(u)$.  To prove the latter, observe that for any subsequence such that $\lim\limits_{k \to \infty} R_{n_k}(u_{n_k}) u_{n_k} = x$ there holds, by (K2) in Definition~\ref{def:Kuratowski} and Theorem~\ref{thm:S_*=S^*}, $x \in S_{*}$.  As, by the continuity of $T$, $\lim\limits_{k \to \infty} T(R_{n_k}(u_{n_k}) u_{n_k}) = T(x)$ and $T(R_{n_k}(u_{n_k}) u_{n_k}) = u_{n_k} \to u$ as $k \to \infty$, we have $T(x) = u$, hence $x = R_{*}(u) u$.  This proves (a), the part of part (b) being similar.

  Since $S_{*}$ equals, by Theorem~\ref{thm:S_*=S^*}, the Kuratowski limit $\widetilde{S}$, it is compact, and as the radial projection $T{\restriction}_{S_{*}} \colon S_{*} \to \Delta$ is a continuous bijection, $S_{*}$ is homeomorphic to $\Delta$.  By Lemma~\ref{lm:properties-S_*}(2), $S_{*}$ is weakly unordered, consequently $S_{*} \in \widehat{\mathcal{U}}$.  The last sentence in part (c) is a consequence of the equality $S_{*} = \widetilde{S}$, Lemma~\ref{lm:properties-tilde-S}(2) and Lemma~\ref{cUintocU}.

  Part (d) is, again in view of $S_{*} = \widetilde{S}$, a consequence of Lemma~\ref{lm:properties-tilde-S}(2).
\end{proof}
We set $\Sigma := S_{*}$.

\medskip

Observe that, by Proposition~\ref{prop:properties-limits}(c)--(d), $\Sigma$ satisfies \ref{P-unordered}, \ref{P-radial-proj} and \ref{P-invariant} in the definition of the carrying simplex [weak carrying simplex].  Notice also that by taking $\varepsilon$ sufficiently small in the definition of $S_0$ we see that $\Sigma$ attracts all points in $\Lambda\setminus \{0\}$.

\begin{theorem}
\label{thm:omega-Lambda}
The compact attractor of bounded sets  $\Gamma = \langle 0, \Sigma \rangle$.
\end{theorem}
\begin{proof}
  \begin{equation*}
    \begin{aligned}
      \omega(\Lambda) & = \bigcap_{n = 0}^{\infty} \cl\Biggl(\bigcup_{k = n}^{\infty} F^k(\langle
      0, S^0 \rangle)\Biggr)
      \\
      & = \bigcap_{n = 0}^{\infty} F^n(\langle 0, S^0 \rangle) & \text{ (by $F^{n+1}(\langle 0, S^0 \rangle) \subset F^{n}(\langle 0, S^0 \rangle)$)}
      \\
      & = \bigcap_{n = 0}^{\infty} \langle 0, S^n \rangle & \text{ (by Lemma~\ref{lm:aux3})}
      \\
      & = \{\, \alpha v : v \in \Delta, 0 \le \alpha \le \inf_{n}R^n(v) \,\}
      \\
      & = \{\, \alpha v : v \in \Delta, 0 \le \alpha \le R_{*}(v) \,\} = \langle 0, \Sigma \rangle.
    \end{aligned}
   \end{equation*}
\end{proof}
As it is straightforward that $\Sigma$ is the relative boundary of $\langle 0, \Sigma \rangle$ in $C_{+}$, the satisfaction of the additional property \ref{P-order-convex} follows from Theorem~\ref{thm:omega-Lambda} together with Lemma~\ref{lm:between}.

\subsection{Back to Dynamics on $C_{+}$}
\label{subsect:back-to-C_+}
Assume \ref{AS-C^0}--\ref{AS4}, and \ref{AS-3} or \ref{AS3-mod} where in case of \ref{AS3-mod} we assume furthermore that any positive number or $\infty$  can serve as $\varkappa$.  From now on, $F$ is considered defined on the whole of $C_{+}$ again.

By Proposition~\ref{prop:0-repeller}, $\{0\} \subset \Gamma$ is a uniform repeller (indeed, Proposition~\ref{prop:0-repeller} states that $\{0\}$ is a uniform repeller for $F{\restriction}_{\Lambda}$; but, as $\Lambda$ is forward invariant, $\{0\}$ is a uniform repeller for $F$, too).

Recall that, by Theorem~\ref{thm:global-attractor-exists}(b), $\Gamma$ is the compact attractor of bounded sets in $C_{+}$, which is the same as the compact attractor of neighbourhoods of compact sets in $C_{+}$.  In the context of Conley's attractor--repeller pairs \cite{conley1} we may decompose $\Lambda$ into an attractor $\Sigma$, a repeller $\{0\}$ and a set of connecting orbits. According to~\cite[Theorem~5.17 on p.~137]{S-T}, the compact attractor $\Gamma$ of neighbourhoods of compact sets in $C_{+}$ is the union of pairwise disjoint sets,
\begin{equation}
\label{eq:repeller-attractor-decomposition}
    \Gamma = \{0\} \cup E \cup H,
\end{equation}
with the following properties:
\begin{itemize}
    \item
    $H$ attracts any bounded $A \subset C_{+}$ with $0 \notin \bar{A}$; further, if $x \in C_{+} \setminus H$ has a bounded backward orbit $\{\, \ldots, x_{-n-1}, x_{-n}, \ldots, x_{-2}, x_{-1}, x\,\}$ then $\lim\limits_{n \to \infty} x_{-n} = 0$;
    \item
    $E$ consists of those $x \in \Gamma$ for which there exists a bounded total orbit $\{\, \ldots, x_{-n-1}, x_{-n}, \allowbreak \ldots, x_{-2}, x_{-1}, x, x_{1}, x_2, \ldots, x_{n}, x_{n+1}, \ldots \,\}$ such that \allowbreak $\lim\limits_{n \to \infty} x_{-n} = 0$ and  $\lim\limits_{n \to \infty} \dist(x_{n}, H) = 0$.
\end{itemize}
We claim that $E = (0, 1) \Sigma$.  Indeed, because $F\!\!\restriction_{\Gamma}$ is a homeomorphism of $\Gamma$ onto itself and $\Sigma$ is invariant, $E \cap \Sigma = \emptyset$.  As $0 \notin E \subset \Gamma$, and $\Gamma \setminus \{0\}$ is the disjoint union of $(0, 1) \Sigma$ and $\Sigma$, there holds $E =  (0, 1) \Sigma$ and $H = \Sigma$.

So we have the following classification of points in $x \in C_{+}$.
\begin{theorem}~
\label{thm:class}
    \begin{enumerate}
        \item[\textup{(a)}]
        If $x \in (0, 1) \Sigma$, then $F^{n}(x) \in (0, 1) \Sigma$ for all $n \in \NN$ and $\lim\limits_{n \to \infty} \dist(F^n(x), \Sigma) = 0$; furthermore, there exists a backward orbit $\{\ldots, x_{-2}, x_{-1}, x\}$ with $\lim\limits_{n \to \infty} x_{-n} = 0$.
        \item[\textup{(b)}]
        If $x \in \Sigma$, then $F^{n}(x) \in \Sigma$ for all $n \in \NN$; furthermore, there exists a backward orbit contained in $\Sigma$.
        \item[\textup{(c)}]
        If $x \in (1, \infty) \Sigma$ then $\lim\limits_{n \to \infty} \dist(F^n(x), \Sigma) = 0$, and there are the following possibilities:
        \begin{itemize}
            \item
            $F^{n}(x) \in (1,\infty) \Sigma$ for all $n \in \NN$;
            \item
            there is $n_0 \in \NN \setminus \{0\}$ such that $F^{n}(x) \in (1,\infty) \Sigma$ for $n = \{0, \ldots, n_0-1\}$ and $F^{n}(x) \in \Sigma$ for $n \ge n_0$;
            \item
            there is $n_0 \in \NN \setminus \{0\}$ such that $F^{n}(x) \in (1,\infty) \Sigma$ for $n = \{0, \ldots, n_0-1\}$ and $F^{n}(x) \in (0, 1) \Sigma$ for $n \ge n_0$;
        \end{itemize}
    \end{enumerate}
\end{theorem}
Further we have
\begin{proposition}
\label{prop:attraction}
    $\Sigma$ attracts any bounded $A \subset C_{+}$ with $0 \notin \bar{A}$.   In other words, $\Sigma$ is the compact attractor of neighbourhoods of compact sets in $C_{+} \setminus \{0\}$.
\end{proposition}
We have thus obtained property \ref{P-attraction}  in the definition of the carrying simplex [weak carrying simplex], as well as the additional properties \ref{P-attractor} and \ref{P-characterization}.

\smallskip
It follows from general results on attractors, see, e.g., \cite[Thms.~2.39 and~2.40]{S-T}, that $\Sigma$, being a compact attractor of neighbourhoods of compact sets, is \textit{stable}, meaning that for each neighbourhood $U$ of $\Sigma$ in $C_{+}$ there exists a neighbourhood $V$ of $\Sigma$ in $C_{+}$ such that $F^n(V) \subset U$ for all $n \in \NN$.  Indeed, by our construction, $\{\, \llangle S_n, S^n \rrangle : n \in \NN \,\}$ is a base of relatively open forward invariant neighbourhoods of $\Sigma$, from which the stability of $\Sigma$ follows in a straightforward way.

\subsection{Asymptotic completeness}
\label{subsect:asymptotic}
Assume \ref{AS-C^0}--\ref{AS4}, and \ref{AS-3} or \ref{AS3-mod} where in case of \ref{AS3-mod} we assume furthermore that any positive number can serve as $\varkappa$.

\begin{proposition}
\label{prop:asymptotic}
    Let additionally \ref{AS4-strong} or \ref{AS-3-strong} hold.  For each $x \in C_{+} \setminus \{0\}$ there exists $y \in \Sigma$ such that $\lim\limits_{n \to \infty} \norm{F^n(x) - F^n(y)} = 0$.
\end{proposition}
\begin{proof}
  As the case $x \in \Sigma$ is obvious, in view of Lemma~\ref{lm:dissipative-1} and Theorem~\ref{thm:class}, by replacing $x$ with some of its iterates, we can assume that either $F^n(x) \in (0, 1) \Sigma$ or $F^n(x) \in (1, \infty) \Sigma \cap \Lambda$, for all $n \in \NN$.

  In the case  $F^n(x) \in (0, 1) \Sigma$ for all $n$, let $A_n := (F^n\!\!\restriction_\Lambda)^{-1} ((F^n(x) + \RR^d_{+}) \cap \Sigma)$.  The sets $A_n$ are compact and, since $(F^n(x) + \RR^d_{+}) \cap \Sigma$ are nonempty, they are nonempty, too.  We claim that $A_{n + 1} \subset A_n$.  Indeed, let $\eta \in A_{n + 1}$, which means that $F^{n + 1}(\eta) \in \Sigma$ and $F^{n + 1}(x) \le F^{n + 1}(\eta)$.  Weak retrotonicity \eqref{AS-weak-retro} gives that $F^{n}(x) \le F^{n}(\eta)$, that is, $\eta \in (F^n\!\!\restriction_\Lambda)^{-1} ((F^n(x) + \RR^d_{+}) \cap \Sigma)$.  The intersection $A := \bigcap\limits_{n = 0}^{\infty} A_n$ is compact and nonempty.  Pick $y \in A$.  By construction, $F^n(x) < F^n(y)$ for all $n \in \NN$.  Suppose to the contrary that $\norm{F^n(x) - F^n(y)} \not\to 0$ as $n \to \infty$.  Take a subsequence $n_k \to \infty$ such that $\lim\limits_{k \to \infty} F^{n_k}(x) = u$ and $\lim\limits_{k \to \infty} F^{n_k}(y) = v$ with $u \ne v$.  Since the $\le$ relation is preserved in the limit, we have $u \le v$, consequently $u < v$.  By Theorem~\ref{thm:class}, both $u$ and $v$ belong to $\Sigma$, which contradicts the unorderedness of $\Sigma$ (Proposition~\ref{prop:properties-limits}(c)).

  In the case $F^n(x) \in (1, \infty) \Sigma \cap \Lambda$ for all $n$, let $A_n := (F^n{\restriction}_\Lambda)^{-1} ([0, F^n(x)] \cap \Sigma)$.  The sets $A_n$ are compact and nonempty.  We claim that $A_{n + 1} \subset A_n$.  Indeed, let $\eta \in A_{n + 1}$, which means that $F^{n + 1}(\eta) \in \Sigma$ and $F^{n + 1}(\eta) \le F^{n + 1}(x)$.  Weak retrotonicity \eqref{AS-weak-retro} gives that $F^{n}(\eta) \le F^{n}(x)$, that is, $\eta \in (F^n{\restriction}_\Lambda)^{-1} ([0, F^n(x)]\cap \Sigma)$.  The intersection $A := \bigcap\limits_{n = 0}^{\infty} A_n$ is compact and nonempty.  Pick $y \in A$.  By construction, $F^n(y) < F^n(x)$ for all $n \in \NN$.  The rest of the proof goes as in the previous paragraph.
\end{proof}

\medskip

Without the additional assumptions stated in Proposition \ref{prop:asymptotic} we have the weaker result.

\begin{proposition}
\label{prop:asymptotic-weak}
    For each $x \in \Gamma \setminus \{0\}$ there exists $y \in \Sigma$ such that $\lim\limits_{n \to \infty} \norm{F^n(x) - F^n(y)} = 0$.
\end{proposition}
\begin{proof}
  Let a nonzero $x \in \Gamma \setminus \Sigma$.  The sets $A_n$ are now defined as $(F^n{\restriction}_\Lambda)^{-1} ((F^n(x) + \RR^d_{+}) \cap \Sigma \cap (C_I)_{+})$, where $I := I(x)$.  As in the proof of Proposition~\ref{prop:asymptotic} we obtain the existence of $y \in \Sigma \cap (C_I)_{+}$ such that $F^n(x) < F^n(y)$ for all $n \in \NN$.  Take a subsequence $n_k \to \infty$ such that $\lim\limits_{k \to \infty} F^{n_k}(x) = u$ and $\lim\limits_{k \to \infty} F^{n_k}(y) = v$. By the closedness of the $\le$ relation, $u \le v$.  For $j \in \{1, \ldots, d\} \setminus I$ we have $u_j = v_j = 0$.

  Let $i \in I$.  By Lemma~\ref{lem:aux1},
    \begin{equation}
     \label{eq:Harnack-infimum}
        \lim_{k \to \infty} h(\pi_{i}(F^{n_k}(x)), \pi_{i}(F^{n_k}(y)))
        = \inf_{n \in \NN} h(\pi_{i}(F^{n}(x)), \pi_{i}(F^{n}(y))).
    \end{equation}
    As a consequence, $0 = u_i < v_i$ is impossible, that is, either $u_i = v_i = 0$ or $0 < u_i \le v_i$.  Suppose $0 < u_i < v_i$.  It follows from~\eqref{eq:Harnack-infimum} with the help of Lemma~\ref{lm:1} that
    \begin{align*}
        h(\pi_{i}(u), \pi_{i}(v)) & = \lim_{k \to \infty} h(\pi_{i}(F^{n_k}(x)), \pi_{i}(F^{n_k}(y)))
        \\
        & = \inf_{n \in \NN} h(\pi_{i}(F^{n}(x)), \pi_{i}(F^{n}(y)))
        \\
        & = \lim_{k \to \infty} h(\pi_{i}(F^{n_k+1}(x)), \pi_{i}(F^{n_k+1}(y))) =  h(\pi_{i}(F(u)), \pi_{i}(F(v))),
    \end{align*}
  But by Lemma~\ref{lem:aux1}, $h(\pi_{i}(F(u)), \pi_{i}(F(v))) < h(\pi_{i}(u), \pi_{i}(v))$, a contradiction.  Since $u = v$ for any convergent subsequence, the statement of the proposition holds.
  \end{proof}
Therefore, \ref{P-asymptotic} is satisfied.  Hence, since now on, $\Sigma$ can be legitimately called the [weak] carrying simplex.

For another proof of Proposition~\ref{prop:asymptotic}, see \cite[Appendix]{RH}, and for another proof of Proposition~\ref{prop:asymptotic-weak}, see \cite[Lem.~5.3]{Hou21}.

\subsection{Lipschitz Property}
\label{subsect:Lipschitz}
We formulate the simple geometrical result:
\begin{lemma}
\label{lm:Lipschitz}
    Let $S \in \widehat{\mathcal{U}}$.  Then
    \begin{enumerate}
        \item[\textup{(a)}]
        $\Pi{\restriction}_{S}$ is injective, consequently, a homeomorphism onto its image;
        \item[\textup{(b)}]
        $\norm{\Pi x - \Pi y} \ge \dfrac{1}{\sqrt{1 + d}} \norm{x - y}$ for any $x, y \in S$.
    \end{enumerate}
\end{lemma}
\begin{proof}
  If $\Pi{\restriction}_{S}$ were not injective, there would be $x, y \in S$ with $y = x + {\alpha} \ev$ for some $\alpha > 0$, which contradicts Lemma~\ref{lm:aux-U}.

  In view of the linearity of $\Pi$ and the fact that no two points in $S$ are in the $\ll$ relation, we will prove (b) if we show that for any nonzero $u \in \RR^d \setminus ( C_{++} \cup ( - C_{++}))$ there holds
  \begin{equation}
  \label{eq:lipschitz}
      \frac{\norm{\Pi u}}{\norm{u}} \ge \dfrac{1}{\sqrt{1 + d}}.
  \end{equation}
  Let $u \in \RR^d$ be arbitrary nonzero.  Again by linearity, we can restrict ourselves to the case when $u = \pm \hat{\ev} + v$, where $v \in V$.  If $\norm{v} < \dfrac{1}{\sqrt{d}}$, then $\norminfty{v} \le \norm{v} < \dfrac{1}{\sqrt{d}}$, so $u \in C_{++} \cup (- C_{++})$.  Consequently, for $u \in \RR^d \setminus ( C_{++} \cup ( - C_{++}))$ we have $\norm{v} \ge \frac{1}{\sqrt{d}}$ and, by the Pythagorean theorem,
  \begin{equation*}
      \frac{\norm{\Pi u}}{\norm{u}} = \frac{\norm{v}}{\norm{\hat{\ev} + v}} = \frac{\norm{v}}{\sqrt{1 + \norm{v}^2}} \ge \dfrac{1}{\sqrt{1 + d}}.
  \end{equation*}
\end{proof}
As $\Sigma \in \widehat{\mathcal{U}}$, the additional property \ref{P-Lipschitz} follows.

\subsection{A sufficient condition for retrotonicity of the map $F$ when $C^1$}
\label{subsect:retrotone}
In the case of discrete\nobreakdash-\hspace{0pt}time models checking whether \ref{AS4} or \ref{AS4-strong} is satisfied may not be an easy task.  In the present subsection we give a simple criterion when $F$ is $C^1$.

 Recall that the spectral radius $\rho(P)$ of a square matrix $P$ is the modulus of the eigenvalue with maximum modulus.
A nonsingular $M$\nobreakdash-\hspace{0pt}matrix is a square matrix $P = \rho_0 I - Z$ where $Z$ is nonnegative and $\rho_0$ exceeds the spectral radius of $\mM$ and that a $P$\nobreakdash-\hspace{0pt}matrix is a square matrix $P$ with positive  principal minors.  It is a standard result that every nonsingular $M$\nobreakdash-\hspace{0pt}matrix is a $P$\nobreakdash-\hspace{0pt}matrix.  Moreover, every eigenvalue of a nonsingular $M$\nobreakdash-\hspace{0pt}matrix has positive real part and every real eigenvalue of a $P$\nobreakdash-\hspace{0pt}matrix is positive \cite{H-J}.

In \cite{G-N} Gale and Nikaid\^o proved an important result on the invertibility of maps whose derivatives are $P$\nobreakdash-\hspace{0pt}matrices on rectangular subsets of $\RR^d$: If $\Omega\subset \RR^d$ is a rectangle and $F \colon \Omega\rightarrow \RR^d$ is a continuously differentiable map such that $\DD F(x)$ is a $P$\nobreakdash-\hspace{0pt}matrix for all $x \in \Omega$ then $F$ is injective in $\Omega$.

Our standing assumption in the present subsection is that $F$ is a Kolmogorov map satisfying  \ref{AS-C1}, \ref{AS-e}, \ref{A-3-diff} and~\ref{GN}, where
\begin{enumerate}
    \myitem[\textup{A1$'$}] \label{AS-C1}
    \quad $f$ is of class $C^1$, with $f(x) \gg 0$ for all $x \in C_+$;
    \myitem[\textup{C}] \label{A-3-diff}
    \quad $\DD f(x) \le 0$ with its diagonal terms negative, for all $x \in C_+$;
    \myitem[\textup{GN}]  \label{GN}
    \quad  The $d\times d$ matrix $\mM(x) = \left(\left(-\frac{x_i}{f_i(x)}\frac{\partial f_i(x)}{\partial x_j}\right)\right)$ has  spectral radius $\rho(\mM(x))<1$ for  all $x \in [0, \ev] \setminus \{0\}$.
\end{enumerate}
Sometimes instead of~\ref{A-3-diff} a stronger assumption is made:
\begin{enumerate}
    \myitem[\textup{C$'$}] \label{A-3-diff-strong}
    \quad $\DD f(x) \ll 0$ for all $x \in C_+$.
\end{enumerate}
We pause a~little to reflect on the $C^1$ property.  The standard definition of a $C^1$ map on $C_{+}$ is that it can be extended to a $C^1$ map on an open subset of $\RR^d$ containing $C_{+}$.  In~fact, since $C_{+}$ is  sufficiently regular, it follows from Whitney's extension theorem (see, e.g., \cite{Robbin}) that the definitions are just the same as in the case of functions defined on open sets, with derivatives replaced by one\nobreakdash-\hspace{0pt}sided derivatives where necessary.

As usual, $\DD F(x)$ denotes the Jacobian matrix of $F$ at the point $x$. If $F$ is invertible with differentiable inverse $F^{-1}$ then by $\DD F^{-1}$ we mean the derivative of $F^{-1}$; this matrix is to be distinguished from $(\DD F)^{-1}$, the matrix inverse of $\DD F$.
\begin{remark}
    In view of \ref{AS-C1}, \ref{A-3-diff} implies \ref{AS-3}, and \ref{A-3-diff-strong} implies \ref{AS-3-strong}.
\end{remark}
\begin{remark}
Observe that \ref{GN} is equivalent to saying that for any $x \in [0, \ev] \setminus \{0\}$ the $\card{I(x)} \times \card{I(x)}$ matrix $\mM(x) := \left(-\frac{x_i}{f_i(x)}\frac{\partial f_i(x)}{\partial x_j}\right)_{i, j \in I(x)}$ has  spectral radius less than $1$.
\end{remark}
\begin{remark}
    Regarding \ref{GN}, observe that for all $x \in C_{+}$ there holds
    \begin{equation}
    \label{eq:Z(x)}
        \DD F(x) = \diag[f(x)] \, (I - Z(x)).
    \end{equation}
\end{remark}
It follows from \ref{GN}, by the continuity of the spectral radius of a matrix, that for any sufficiently small $\varkappa > 0$ there holds $\rho(x) < 1$ for all $x \in [0, (1 + \varkappa) \ev] \setminus \{0\}$.

We choose such a $\varkappa > 0$, and put $\Lambda := [0, (1 + \varkappa) \ev]$.

\begin{lemma}
\label{lm:inverse}
    For $x \in \Lambda$, $(\DD F(x))_{ij}^{-1} \ge  0$ for all $i, j \in \{1, \ldots, d\}$; moreover, for $x \in \Lambda'$,
    \begin{itemize}
        \item
        $(\DD F(x))^{-1}_{ii} > 0$ for all $i \in I(x)$;
        \item
        Under \ref{A-3-diff-strong}, $(\DD F(x))^{-1}_{ij} > 0$ for all $i, j \in I(x)$.
        \item
        $\DD F_{ij}(x) \leq  0$ for $i\neq j$ and $\DD F_{ii}(x) > 0$ for $i=1,\ldots,d$.  Under~\ref{A-3-diff-strong}, $\DD F_{ij}(x) < 0$ for $i \ne j$, provided $x_i > 0$.
    \end{itemize}
\end{lemma}
\begin{proof}
  Let $x \in \Lambda'$. Then $\mM(x)$ is a nonnegative matrix with $\rho(\mM(x)) < 1$.    $P(x) := I - \mM(x)$ is an $M$\nobreakdash-\hspace{0pt}matrix, and also a $P$\nobreakdash-\hspace{0pt}matrix for $x \in \Lambda$. Since, by~\eqref{eq:Z(x)}, $\DD F = \diag[f](I-\mM)$, $\rho(Z(x))<1$, $Z(x)_{ij} \ge 0$ for $i, j \in I(x)$, and  $f(x)\gg 0$ we find that $(\DD F(x))^{-1} = (I - \mM (x))^{-1} \diag [f(x)]^{-1} = \left(\sum\limits_{k=0}^\infty \mM (x) ^k\right)  \diag[f(x)]^{-1} \geq 0$ for $x\in \Lambda$, and $(\DD F(x))^{-1}_{ii} > 0$ for $i \in I(x)$  [under \ref{A-3-diff-strong}, $(\DD F(x))^{-1}_{ij} > 0$ for $i, j \in I(x)$].

 Lastly, since $\DD F(x)$ is a nonsingular $M$\nobreakdash-\hspace{0pt}matrix its diagonal elements must be positive and its off\nobreakdash-\hspace{0pt}diagonal elements must be non-positive.  The final sentence follows directly from~\ref{AS-C1} and~\ref{A-3-diff-strong} by the form of $Z(x)$.
\end{proof}

\begin{proposition}~
\label{lem1}
  \begin{enumerate}
      \item[\textup{(a)}]
      $F{\restriction}_{\Lambda}$ is a $C^1$\nobreakdash-\hspace{0pt}diffeomorphism onto its image.
      \item[\textup{(b)}]
      $F$ is weakly retrotone in $\Lambda$. If \ref{A-3-diff-strong} holds, then $F$ is  retrotone in $\Lambda$.
  \end{enumerate}
\end{proposition}
\begin{proof}
 By Lemma~\ref{lm:inverse}, $\DD F(x)$ is invertible at each $x \in \Lambda$, so the inverse function theorem implies that $F{\restriction}_{\Lambda}$ is a local $C^1$ diffeomorphism.  Applying Lemma~\ref{lm:AS-homeo} gives us that $F{\restriction}_{\Lambda}$ is indeed a $C^1$ diffeomorphism.  For an alternative approach, see~\cite[Thm.~4]{G-N}.

We proceed to the proof of part (b).  We will prove weak retrotonicity by induction on the dimension $d$.  For $d = 1$ it follows from \ref{GN} that $F$ is increasing on the segment $[0, 1+\varkappa]$, so the required property holds. Now, let $d > 1$ and assume that weak retrotonicity holds for any system satisfying \ref{AS-C1}, \ref{A-3-diff} and~\ref{GN}, of dimension $< d$.  Suppose to the contrary that there exist $x, y \in \Lambda$ and (up~to possible relabelling) $1 \le m \le d$ such that
\begin{equation*}
    F(x) \le F(y) \quad \text{but} \quad
    \begin{cases}
        x_i > y_i & \text{ for } i = 1, 2, \ldots, m
        \\
        x_i \le y_i & \text{ for } i = m + 1, \ldots, d.
    \end{cases}
\end{equation*}
First suppose $m=d$, i.e. $F(x)\leq F(y)$ but $x\gg y$. As noted in the proof of Lemma \ref{lm:inverse}, under the stated conditions $DF$ is a $P$-matrix on $\Lambda$. Thus by~\cite[Thm.~3]{G-N}, for $x, y \in \Lambda$, the inequalities $F(x) \le F(y)$ and $x \ge y$ only have the solution $x=y$, and so the case $m = d$ is not possible.

This leaves the case $1\leq m < d$.
We identify the subspace $\{\, (\xi_1, \ldots, \xi_m, 0, \ldots, 0) \,\}$ with $\RR^m$.  Define $H \colon \Lambda \cap \RR^m \to \RR^m$, $H = (H_1, \ldots, H_m)$, by
\begin{equation*}
    H_i(\alpha_1, \ldots, \alpha_m) := F_i(\alpha_1, \ldots, \alpha_m, y_{m+1}, \ldots, y_d).
\end{equation*}
Observe that $H$ satisfies all the conditions imposed on $F$, in~particular $\DD H$ has positive diagonal entries and non-positive off\nobreakdash-\hspace{0pt}diagonal entries.  For $i = 1, 2, \ldots, m$ one has
\begin{multline*}
    H_i(y_1, \ldots, y_m) = F_i(y) \ge F_i(x) = F_i(x_1, \ldots, x_m, x_{m+1}, \ldots, x_d)
    \\
    \ge F_i(x_1, \ldots, x_m, y_{m+1}, \ldots, y_d) = H_i(x_1, \ldots, x_m),
\end{multline*}
where the second inequality holds because $x_{m+1} \le y_{m+1}, \dots, x_d \le y_d$ and $F_i$ is non\nobreakdash-\hspace{0pt}increasing in its $x_{m+1}, \ldots, x_d$ arguments.  But this contradicts our inductive assumption.
\end{proof}

It should be mentioned that Proposition~\ref{lem1} appeared as \cite[Lem.~5.1]{Hou21}, and, in the case of \ref{A-3-diff-strong}, as \cite[Prop.~1.1]{RH}.

\smallskip
We collect what we have just proved as
\begin{proposition}~
\label{prop:GN}
    \begin{itemize}
        \item
        \ref{AS-C1}, \ref{AS-e}, \ref{A-3-diff} and~\ref{GN} imply \ref{AS-C^0}--\ref{AS-3},
        \item
        \ref{AS-C1}, \ref{AS-e}, \ref{A-3-diff-strong} and~\ref{GN} imply \ref{AS-C^0}--\ref{AS-3}, \ref{AS4-strong} and \ref{AS-3-strong}.
    \end{itemize}
\end{proposition}

\section{Examples}
\label{sect:examples}
The purpose of the present section is to give simple examples relating to well-known models from theoretical ecology to illustrate the applicability of our results.

The examples given in this section cover the situation when $F$ is given by some closed\nobreakdash-\hspace{0pt}form formula.  Direct checking whether \ref{AS4} is satisfied appeared to be a hopeless task, so we use \ref{GN} instead.  On the other hand, it is easy to check~\ref{A-3-diff} or \ref{A-3-diff-strong}.  In the case of Beverton--Holt (Subsection~\ref{subsub:Beverton-Holt}) or Atkinson--Allen (Subsection~\ref{subsect:Atkinson-Allen}) $F$ is a (weakly) retrotone homeomorphism on the whole of $C_{+}$, whereas for the Ricker case (Subsection~\ref{subsect:2DRicker}) the `box' $\Lambda$ on which $F$ is a (weakly) retrotone homeomorphism cannot be too large.

%\subsection{One-dimensional examples}

\subsection{$d=1$: Beverton--Holt map}
\label{subsub:Beverton-Holt}
%\label{subsect:1D}

Let $n = 1$.  Observe that our assumptions have now the following meaning:
\begin{itemize}
[align=left]
\item[\ref{AS-C1}$_{\!\! 1}$]
\quad $f$ is of class $C^1$ on $[0, \infty)$, and $f(x) > 0$ for all $x \ge 0$;
\item[\ref{AS-e}$_1$] \quad  $f(1)= 1$;
\item[\ref{A-3-diff-strong}$_{\!\!1}$] \quad  $f'(x) < 0$ for all $x \ge 0$;
\item[\ref{GN}$_1$] \quad  $F'(x) = x f'(x) + f(x) > 0$ for $x \in (0, 1]$.
\end{itemize}
Take $f(x) = \frac{2}{1+x}$, so that $F(x) = \frac{2x}{1+x}$. The map $F$ has fixed points only at $x = 0, 1$. \ref{AS-C1}$_{\!\! 1}$, \ref{AS-e}$_1$ are easily checked and $f'(x) = - \frac{2}{(1+x)^2} < 0$ shows \ref{A-3-diff-strong}$_{\!1}$ is satisfied on $C_+ = [0, \infty)$.

Finally,
\begin{equation*}
    Z(x) = -\frac{x}{f(x)} f'(x) = \frac{x}{x+1} < 1, \quad x \in [0, \infty),
\end{equation*}
equivalently
\begin{equation*}
    F'(x) = - \frac{2 x}{(1 + x)^2} + \frac{2}{1 + x} = \frac{2}{(1 + x)^2} > 0, \quad x \in [0, \infty).
\end{equation*}
$F$ is a homeomorphism from $[0, \infty)$ onto $[0,2)$ and we may take $\Lambda = [0,1 + \varkappa]$, where $\varkappa > 0$ is arbitrary. The fixed point $x=1$ is the carrying simplex and it is easy to check that it attracts any $A \subset (0, \infty)$ with $0 \notin \bar{A}$.

Any $x \in (0, 1)$ has a unique backward orbit $\{\ldots, x_{-n-1}, x_{-n}, \ldots, x_{-1}, x\}$ strictly decreasing, as $n \to \infty$, to $0$.  Its forward orbit strictly increases, as $n \to \infty$, to $1$.

The forward orbit of any $x > 1$ strictly decreases to $1$.  No $x > 1$ has an (infinite) backward orbit, however for any $x \in (0, 2)$ there are $n_0 \ge 1$ and $x = F(x_{-1}) < x_{-1} = F(x_{-2}) < \ldots < x_{-n_0+1} = F(x_{-n_0}) < x_{-n_0}$.

\subsection{$d=2$: Planar Ricker map}
\label{subsect:2DRicker}
Consider the following planar Ricker model:
\begin{equation}
F(x,y) = (x e^{r(1-x-ay)},ye^{s(1-y-bx)}) \label{ricker}
\end{equation}
where $r, s > 0$, $a, b \geq 0$. Thus
\[
f(x,y) = (e^{r(1-x-ay)},e^{s(1-y-bx)}).
\]
The map $F \colon C_{+} \rightarrow C_{+}$ where $C_{+} = [0, \infty)^2$, always has $(0,0)$, $\ev_1 = (1,0)$ and $\ev_2 = (0,1)$ as fixed points, so \ref{AS-C1} and \ref{AS-e} are satisfied.
\begin{equation*}
    \DD F(x,y) =
    \left(\begin{array}{cc}
    e^{r(1-x-ay)}(1-rx) & -arx e^{r(1-x-ay)}
    \\
    -sb ye^{s(1-y-bx)}& e^{s(1-y-bx)}(1-sy)
    \end{array} \right)
\end{equation*}
and $Z(x,y) = \left(\begin{array}{cc}rx & arx \\
sb y&sy\end{array}
\right)
$.

Under our assumptions \ref{A-3-diff} is satisfied, and \ref{A-3-diff-strong} holds provided $a > 0$ and $b > 0$.  The eigenvalues $\lambda$ of $Z(x,y)$ satisfy $\lambda^2-(rx+sy)\lambda+rsxy(1-ab)=0$, so by the Jury condition the matrix $Z(x,y)$ has spectral radius less than one when $(x,y)\in C_{+}$ satisfies
\begin{equation}
rx + sy < 1 + rsxy(1-ab) < 2. \label{rick1}
\end{equation}
The bounded connected component $\Omega \subset C_{+}$ of the set defined by \eqref{rick1} is a simply connected region whose closure contains the points $(0,0), (1/r,0),(0,1/s)$. For the Ricker model to satisfy \ref{GN} on $[0,1]^2$ we require $[0,1]^2\subset \Omega$. Thus natural conditions are
\[
r,s < 1 \quad \text{and} \quad r+s < 1+rs(1-ab) < 2
\]
(since $\ev_1$, $\ev_2$, $(1,1)$ must belong to $\bar{\Omega}$).  Indeed, the former condition is redundant: $r + s < 1 + rs(1 - ab)$ is equivalent to $(1 - r) (1 - s) > r s a b$, from which it follows that $r$ and $s$ are either both $ > 1$ or both $< 1$, and $r + s < 2$ gives $r, s < 1$.

Since the boundary of $\Omega$ in $C_{+}$ is the graph of a decreasing function we can take $\Lambda = [0, 1 + \varkappa]^2$, where $\varkappa > 0$ is so small that
\begin{equation*}
   (1 + \varkappa) (r + s) < 1 + (1 + \varkappa)^2 r s (1 - a b) < 2.
\end{equation*}
In view of Proposition~\ref{lem1}(b) we have the following.
\begin{lemma}
  For the Ricker map \eqref{ricker} assumptions \ref{A-3-diff} and \ref{GN} are satisfied if $r, s > 0$, $a, b \ge 0$ and
 \begin{equation}
 r+s < 1+rs(1-ab)<2 \label{rick2}
\end{equation}
If $r, s, a, b > 0$ and \eqref{rick2} holds then \ref{A-3-diff-strong} and \ref{GN} are satisfied, so that there is a carrying simplex consisting of a curve that connects the two axial fixed points $(r,0)$ and $(0,s)$.
\end{lemma}

\begin{figure}[htb]
\centerline{
\includegraphics[width=2.5in]{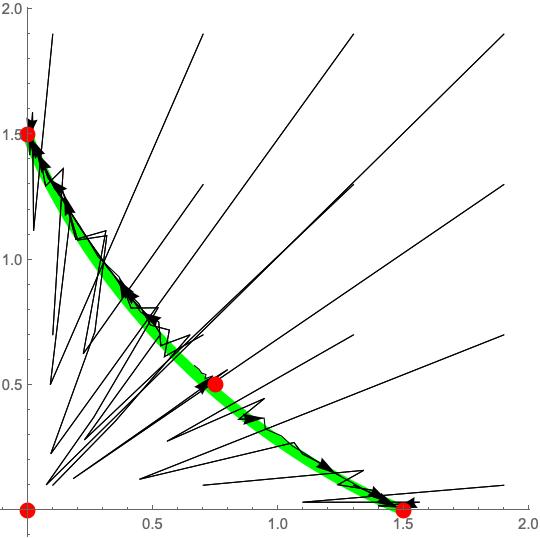}
}
\caption{The carrying simplex (green solid line) for the planar Ricker map (\ref{ricker}) when $a=3/2$, $b=4/3$, $r=s=3/2$.}
\label{fig1}
\end{figure}

\subsection{Arbitrary $d$:  Atkinson--Allen map}
\label{subsect:Atkinson-Allen}
The Atkinson--Allen map  is given by
\begin{equation}
\label{AAm}
    F_i(x) = b x_i+ \frac{2(1-b)x_i}{1+(Ax)_i}, \;\; i=1,\ldots,d,
\end{equation}
where $b \in (0,1)$ and  $A > 0$. A more slightly general map, which includes \eqref{AAm}, called the generalized competitive Atkinson--Allen map, was studied in \cite{Gyll2} where the map was shown to have a carrying simplex for a large range of positive parameter values. In \cite{Hou21}, these results are extended to cover cases where some parameters are zero.

If the range of $b$ is extended to include $b=0$ the map \eqref{AAm} reduces to a Leslie--Gower map, and this is known to have a carrying simplex for all $A\gg 0$ \cite{Gyll2} and all $A > 0$ with $a_{ii}>0$ when $i=1,\ldots,d$ \cite{Hou21}. If the range of $b$ is extended to include $b=1$, $F$ is the identity map and there is no carrying simplex.

For (\ref{AAm}) we have
\[
f_i(x)= b+ \frac{2(1-b)}{1+(Ax)_i}, \;\; i=1,\ldots,d.
\]
$f$ is continuously differentiable on $C_+$ and $f_i(x)\geq b>0$ so that assumption \ref{AS-C1}. is satisfied. It is also easy to check that $f_i(\ev_i)=1$ so assumption \ref{AS-3} is satisfied.
As
\[
\frac{\partial f_i(x)}{\partial x_j} = \left(\frac{-2(1-b)a_{ij}}{(1+(Ax)_i)^2}
\right) < 0 \; \text{ as } b\in (0,1)
\]
so that assumption \ref{A-3-diff-strong} is satisfied. Now $\rho(Z(x)) = \rho(\diag[x]^{-1} \, Z(x) \, \diag[x])\leq \norm{\diag[x]^{-1} \, Z(x) \, \diag[x]}$ for any matrix norm. Let us choose the $\ell_\infty-$norm, so that
\begin{align*}
\rho(Z(x))  & \leq \norm{\diag[x]^{-1}\, Z(x) \, \diag[x])}_\infty \\
& = \max_i \sum_{j=1}^3 \frac{1}{x_i}Z_{ij}(x) x_j\\
& =  \max_i \sum_{j=1}^d -\frac{x_j}{f_i}\frac{\partial f_i}{\partial x_j}\\
& = \max_i \sum_{j=1}^d -\frac{x_j}{b+ \frac{2(1-b)}{1+(Ax)_i}} \left(\frac{-2(1-b)a_{ij}}{(1+(Ax)_i)^2}
\right)\\
= &2(1-b)\max_i  \frac{(Ax)_i}{b+\frac{2(1-b)}{1+(Ax)_i}}\frac{1}{(1+(Ax)_i)^2}\\
\leq & 2(1-b) \max_{s\geq 0} \frac{s}{(1+s)(b(1+s) +2(1-b)}\\
=& \frac{1-\sqrt{(2-b) b}}{1-b} \in (0,1) \mbox{ when } b\in (0,1),
\end{align*}
which confirms assumption \ref{GN}. Hence the Atkinson--Allen map has a carrying simplex for all values of $b\in(0,1)$ and $A > 0$. Figure \ref{figAA} shows an example of the carrying simplex for the Atkinson--Allen map (\ref{AAm}) when $d=3$ and $A=\left( \begin{array}{ccc} 1& 1/2 & 1/3 \\ 1/3 & 1 & 1/2 \\ 1/2 & 1/3 &1 \end{array}\right)$   together with a selection of orbits, all of which are attracted to the carrying simplex.
\begin{figure}[htb]
\centerline{
\includegraphics[width=2.5in]{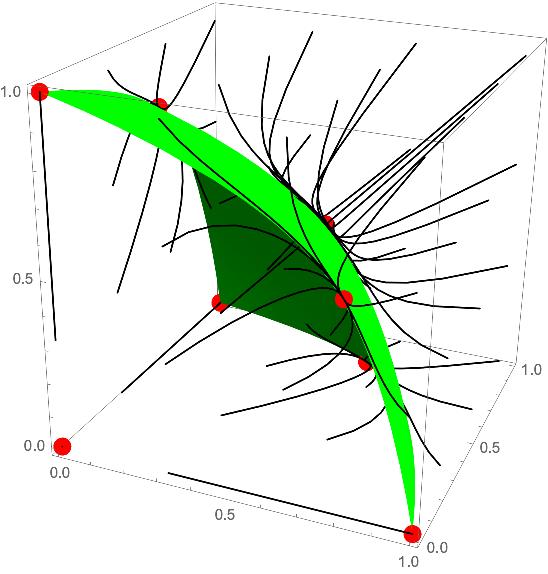}
}
\caption{The carrying simplex (green surface) for the Atkinson--Allen ($d=3$) map (\ref{AAm}) when $b=1/4$ and $A=\left( \begin{array}{ccc} 1& 1/2 & 1/3 \\ 1/3 & 1 & 1/2 \\ 1/2 & 1/3 &1 \end{array}\right)$.}
\label{figAA}
\end{figure}

\section{An application to competitive systems of ODEs}
\label{subsect:ODEs}
This section  is devoted to the case when $F$ is a time\nobreakdash-\hspace{0pt}one map of the semiflow generated by an autonomous competitive system of ODEs.  Even when the vector field is given by some formula, it is seldom possible to find a formula for $F$.  We have at our disposal, however, the Müller--Kamke  theory \cite{muller,kamke}, which allows us to show that $F$ is a (weakly) retrotone homeomorphism onto its image.  Further, \ref{AS3-mod} is a consequence of the competitivity property of the ODE system.

\medskip

Consider an autonomous system of ordinary differential equations of the form
\begin{equation}
\label{eq:ODE}
    \frac{\dd x_i}{\dd t} = x_i g_i(x), \quad i \in \{1, \ldots, d \}, \ x = (x_1, \ldots, x_d) \in C_{+},
\end{equation}
where $g = (g_1, \ldots, g_d) \colon  C_{+} \to \RR^d$.  Such systems are called \emph{Kolmogorov systems of ODEs}.

We denote $P = \diag[\Id] \, g$, $P = (P_1, \dots, P_d)$.  $\mathrm{D} g$ denotes the derivative matrix of $g$ (if it exists).

The first assumption is
\begin{enumerate}[label=\textup{H\arabic*},ref=\textup{H\arabic*},align=left]
    \item \label{ODE-C^1}
    \quad $g$ is of class $C^1$.
\end{enumerate}
For $x \in C_{+}$ denote by $\Phi(\cdot; x)  = (\Phi_1(\cdot; x), \dots, \Phi_d(\cdot; x))$ the (unique by \ref{ODE-C^1}) nonextendible solution of \eqref{eq:ODE} taking value $x$ at time $0$.

Observe that if $x \in (C_I)_{++}$ for some $I \subset \{1, \dots, d\}$ then $\Phi(t; x) \in (C_I)_{++}$ as long as it exists. Indeed, it follows from the Kolmogorov form of~\eqref{eq:ODE} and the uniqueness of solutions that $\Phi_j(t; x) = 0$ for all $j \in \{1, \dots, d\} \setminus I$.

As a consequence, since $(C_I)_{++}$ is a relatively open subset of $(C_I)_{+}$, it follows from the extension theorem for ODEs that the nonextendible solution $\Phi(t; x)$ is defined for $t \in (\taumin(x), \taumax(x))$ with $\taumin(x) < 0 < \taumin(x)$.

The local flow (continuous-time dynamical system) $\Phi$ on a subset $D(\Phi)$ of  $\RR \times C_{+}$ satisfies the following properties:
\begin{enumerate}[label=\textup{(EF\arabic*)},ref=\textup{(EF\arabic*)},align=left]
\setcounter{enumi}{-1}
    \item \label{EF0}
    $D(\Phi)$ is an open subset of $\RR \times C_{+}$ containing $\{\, (0, x) : x \in C_{+}\,\}$, and $\Phi$ is $C^1$;
    \item \label{EF1}
    $\Phi(0; x) = x$ for all $x \in C_{+}$;
    \item \label{EF2}
    \begin{equation*}
        \Phi(t_2; \Phi(t_1; x)) = \Phi(t_1 + t_2; x),
    \end{equation*}
    which is to be interpreted so that if for some $t_1, t_2 \in \RR$ and $x \in C_{+}$ one of the sides as well as $\Phi(t_1, x)$ exist then the other side exists, too, and the equality holds.
\end{enumerate}

\medskip
By \ref{EF1}--\ref{EF2}, we have the formula
\begin{equation}
\label{eq:process-inverse}
    \Phi(- t, \Phi(t; x)) = x,
\end{equation}
provided that $\Phi(t; x)$ exists.  In~particular, it follows that for each $t > 0$ the map $\Phi(t; \cdot)$ is a $C^1$ diffeomorphism onto its image.

We will give now a representation of its inverse in terms of a solution of system~\eqref{eq:ODE} with $P$ replaced with $- P$.  For $x \in C_{+}$ and fixed $t \in \RR$ such that $\Phi(t; x)$ exists we put $\phi(s) := \Phi(s; x)$, $s \in (\taumin(x), \taumax(x))$.  Let $\chi(\theta) := \phi(t - \theta)$.  There holds, where $\dot{} = \frac{\mathrm{d}}{\mathrm{d}s}$,
\begin{equation*}
    \frac{\mathrm{d}\chi}{\mathrm{d}\theta}(\theta) = - \dot{\phi}(t - \theta) = - P(\phi(t -\theta)) = - P(\chi(\theta)), \qquad \theta \in (t - \taumax(x), t - \taumin(x)).
\end{equation*}
We formulate the result of the above calculation as the following.
\begin{lemma}
\label{lm:inverse-ODE}
    Assume that $x \in C_{+}$ and $t \in \RR$ are such that $\Phi(t; x)$ exists.  Then $x$ is equal to the value at time $\theta = t$ of the solution of the initial value problem
\begin{equation}
\label{eq:ODE-retro}
     \begin{cases}
         \displaystyle \frac{\mathrm{d}\xi_i}{\mathrm{d}\theta} = - \xi_i g_i(\xi), \quad i \in \{1, \ldots, d \}, \ \xi \in C_{+},
         \\[1.5ex]
         \xi(0) = \Phi(t; x).
     \end{cases}
\end{equation}
\end{lemma}

The next assumptions are:
\begin{enumerate}[label=\textup{H\arabic*},ref=\textup{H\arabic*},align=left]
\setcounter{enumi}{1}
    \item \label{ODE-axes}
    \quad for each $1 \le i \le d$ there holds $g_i(\ev_i) = 0$;
    \item \label{ODE-competitive}
    \quad $\mathrm{D}g(x) \le 0$ with its diagonal entries negative, for all $x \in C_{+}$.
\end{enumerate}

Sometimes instead of \ref{ODE-competitive} we make the following stronger assumption:
\begin{enumerate}[label=\textup{H\arabic*$'$},ref=\textup{H\arabic*$'$},align=left]
\setcounter{enumi}{2}
    \item \label{ODE-competitive-strong}
    \quad $\mathrm{D}g(x) \ll 0$, for all $x \in C_{+}$.
\end{enumerate}

We deduce from \ref{ODE-axes} and the negativity of the diagonal terms in \ref{ODE-competitive} that the vector field restricted to an $i$\nobreakdash-\hspace{0pt}axis takes value $0$ at $0$ and $\ev_i$, has positive direction between $0$ and $\ev_i$ and negative direction to the right of $\ev_i$.  In view of that, from the nonpositivity property of the off\nobreakdash-\hspace{0pt}diagonal entries it follows that if $x \in \prod\limits_{i=1}^d [0, 1 + \varkappa]$ then $\Phi(t; x) \in \prod\limits_{i=1}^d [0, 1 + \varkappa]$ for $t \ge 0$ as long as $\Phi(t; x)$ exists, for any $\varkappa > 1$.  In~particular, from the standard ODEs extension theorem we obtain that $\Phi(t; x)$ exists for any $t \ge 0$ and any $x \in C_{+}$.

\medskip

We now set the map $F(\cdot) = \Phi(1; \cdot)$.

\medskip

For $x \in C_{++}$ and $i \in \{1, \ldots, d\}$ we write
\begin{equation}
\label{eq:ODE-f}
    f_i(x) = \frac{F_i(x)}{x_i} = \exp{\biggl(\int_{0}^{1} g_i(\Phi(\tau; x)) \, \mathrm{d}\tau \biggr)}.
\end{equation}
By the continuous dependence of solutions of ODEs on initial values and the continuity of $g_i$, the formula~\eqref{eq:ODE-f} for $f_i$ extends to the whole of $C_{+}$.  Since $F$ is continuous on $C_{+}$, we have that $F = \diag[\Id] \, f$ with $f$ given by~\eqref{eq:ODE-f} on $\partial C_{+}$, too.  Therefore \ref{AS-C^0} is fulfilled.

The fulfillment of \ref{AS-e} follows directly from \ref{ODE-axes}.

As $F(\cdot) = \Phi(1; \cdot)$ is a diffeomorphism onto its image, \ref{AS-homeo} is satisfied.

\begin{lemma}
\label{lm:ODE-retrotone}
    For any $t > 0$, $\Phi(t; \cdot)$ is weakly retrotone in $C_{+}$, and, under~\ref{ODE-competitive-strong}, is retrotone in $C_{+}$.
\end{lemma}
\begin{proof}
  Since $\Phi(t; \cdot)$ is injective, proving the statement is just showing monotonicity of its inverse on the faces of $C_{+}$.  This, by Lemma~\ref{lm:inverse-ODE}, follows along the lines of the proof of~\cite[Prop.~2.2]{HirschCS}, which is in turn a consequence of the Müller--Kamke theorem, see, e.g., \cite[Thm.~2]{Wal} and \cite[Thm.~4]{Wal} for \ref{ODE-competitive-strong}.
\end{proof}
We have thus obtained \ref{AS-weak-retro}, and \ref{AS-retro} for \ref{ODE-competitive-strong}.
Consequently, \ref{AS4} is satisfied, and, under \ref{ODE-competitive-strong}, \ref{AS4-strong} is satisfied, for all $\varkappa > 0$.

\medskip
We proceed now to proving \ref{AS3-mod}.  For $x, y \in C_{+}$ we have, by~\eqref{eq:ODE-f}, for any $i \in \{1, \ldots, d \}$,
\begin{equation*}
    \frac{f_i(x)}{f_i(y)} = \exp\biggl(\int_{0}^{1} \bigl(g_i(\Phi(s, x) - g_i(\Phi(s,y)\bigr) \, \mathrm{d}s\biggr).
\end{equation*}
For each $s \in [0, 1]$ there holds
\begin{equation}
\label{eq:aux}
    g(\Phi(s, x)) - g(\Phi(s,y)) = \biggl( \int_{0}^{1} \mathrm{D} g\bigl(\theta \Phi(s, x) + (1 - \theta) \Phi(s, y)\bigr) \, \mathrm{d}\theta \biggr) (\Phi(s, x) - \Phi(s, y)),
\end{equation}
It follows from \ref{ODE-competitive} that the matrix on the right\nobreakdash-\hspace{0pt}hand side of~\eqref{eq:aux} has nonpositive entries, and negative diagonal entries.  Under~\ref{ODE-competitive-strong}, that matrix has all entries negative.

Assume now that $x, y \in C_{+}$ are such that $F(x) < F(y)$.  Applying Lemma~\ref{lm:ODE-retrotone} and comparing the signs of the entries/coordinates on the right\nobreakdash-\hspace{0pt}hand side of~\eqref{eq:aux} gives the desired inequalities for
$g_i(\Phi(s, x)) - g_i(\Phi(s,y))$, $s \in [0, 1)$.

\medskip

We can thus apply Theorem~\ref{prop1_steve} to obtain the existence of the carrying simplex [weak carrying simplex] for $F$:
\begin{theorem}
\label{ODE_CS}
Under the assumptions \ref{ODE-C^1}, \ref{ODE-axes} and  \ref{ODE-competitive} \textup{[}resp.\ \ref{ODE-competitive-strong}\textup{]} the competitive system of ordinary differential equations \eqref{eq:ODE} has a weak carrying simplex \textup{[}resp.\ carrying simplex\textup{]}.
\end{theorem}

\section*{Acknowledgements}
We are grateful to the referees for their insightful remarks.

\section*{Funding}
J. Mierczyński is partially supported by the Faculty of Pure and Applied Mathematics, Wrocław University of Science and Technology.

\section*{Disclosure statement}
The authors report there are no competing interests to declare.

\end{document}